\documentclass[11pt,final]{article}
\usepackage{mathtools}\mathtoolsset{showonlyrefs,showmanualtags,mathic=true}
\usepackage{unicode-math}
\usepackage{marvosym}   %for letter symbol
\usepackage{fourier-orns}   %for ornament
\usepackage{tikz}
\usepackage{tikz-cd}
  \tikzcdset{arrow style=tikz, diagrams={>=stealth}}
  \usetikzlibrary{calc}
  \usetikzlibrary {arrows.meta}
\usepackage[dvipsnames]{xcolor}
\usepackage{fullpage}
\usepackage{caption}
\usepackage{enumitem}
   \setlist{leftmargin=0pt,itemindent=2.5em,labelsep=1em}
\usepackage[breaklinks=true,colorlinks=true,allcolors=black]{hyperref}
\usepackage[color]{showkeys}
  \definecolor{refkey}{rgb}{0.9451,0.2706,0.4941}
  \definecolor{labelkey}{rgb}{0.9451,0.2706,0.4941}
\usepackage{microtype}
\usepackage{tabularray}
\usepackage[title,titletoc]{appendix}

\allowdisplaybreaks 
\numberwithin{equation}{section}

\usepackage{amsthm}
 \theoremstyle{definition}
   \newtheorem{definition}{Definition}[section]
   \newtheorem{definitions}[definition]{Definitions}
   \newtheorem{example}[definition]{Example}
   
 \theoremstyle{plain}
   \newtheorem{theorem}[definition]{Theorem}
   \newtheorem{maintheorem}[definition]{Main Theorem}
   \newtheorem{proposition}[definition]{Proposition}
   \newtheorem{definition-proposition}[definition]{Definition-Proposition}
   \newtheorem{lemma}[definition]{Lemma}
   \newtheorem{corollary}[definition]{Corollary}
   
   \newtheorem{conjecture}[definition]{Conjecture}
 \theoremstyle{remark}
   \newtheorem{remark}[definition]{Remark}

\renewcommand{\emph}[1]{\textup{\textbf{#1}}}

\DeclareMathOperator{\Aut}{Aut}
\DeclareMathOperator{\Gal}{Gal}
\DeclareMathOperator{\ord}{ord}
\DeclareMathOperator{\GL}{GL}
\DeclareMathOperator{\tr}{tr}
\DeclareMathOperator{\im}{im}
\DeclareMathOperator{\SI}{SI}

\title{
\begin{tabular}{c}Distributions of Iwasawa \(\lambda\)-invariants \\ of \(\symbf{Z}_p\)-towers over supersingular isogeny graphs
\end{tabular}
}
\author{
Taiga Adachi\thanks{Joint Graduate School of Mathematics for Innovation, Kyushu University, Fukuoka, Japan\\
\phantom{AAAA}\Letter\ \href{mailto:t.adachi1729@gmail.com}{t.adachi1729@gmail.com}}
\and Kosuke Mizuno\thanks{Graduate School of Mathematics, Nagoya University, Nagoya, Japan\\
\phantom{AAAA}\Letter\ Kosuke Mizuno: \href{mailto:kosuke.mizuno.c1@math.nagoya-u.ac.jp}{kosuke.mizuno.c1@math.nagoya-u.ac.jp}\\
\phantom{AAAA}\Letter\ Ryosuke Murooka: \href{mailto:ryosuke.murooka.c1@math.nagoya-u.ac.jp}{ryosuke.murooka.c1@math.nagoya-u.ac.jp}}
\and Ryosuke Murooka\footnotemark[2]\\ 
\and Sohei Tateno\thanks{Graduate School of Natural Science \& Technology, Kanazawa University, Kanazawa, Japan \\
\phantom{AAAA}\Letter\ \href{mailto:inu.kaimashita@gmail.com}{inu.kaimashita@gmail.com}}
}
\usepackage[automake,toc=false,nopostdot,style=long3colheader,sort=use]{glossaries-extra}
\makeglossaries
  \newglossaryentry{vertices}{
   name={\ensuremath{\symbf{V}(X)}},
   description={The set of vertices of a (directed) graph \(X\).}
  }
  \newglossaryentry{edges}{
   name={\ensuremath{\symbf{E}(X)}},
   description={The set of edges of a (directed) graph \(X\).}
  }
  \newglossaryentry{adjacencymatrix}{
   name={\ensuremath{A(X)}},
   description={The adjacency matrix of a (directed) graph \(X\).}
  }
  \newglossaryentry{degreematrix}{
   name={\ensuremath{D(X)}},
   description={The degree matrix of a (directed) graph \(X\).}
  }
  \newglossaryentry{supersingularisogeny}{
  name={\ensuremath{\SI(r,\ell)}},
  description={The supersingular \(\ell\)-isogeny graph over \(\wideoverbar{\symbf{F}}_r\).}
  }
  \newglossaryentry{doublesupersingularisogeny}{
  name={\ensuremath{X^{(r,\ell)}}},
  description={The double supersingular \(\ell\)-isogeny graph corresponding to \(\SI(r,\ell)\).}
  }
  \newglossaryentry{newforms}{
  name={\ensuremath{\symscr{B}_2(r)}},
  description={The set of newforms of weight 2 and level \(r\).}
  }
  \newglossaryentry{characteristicelement}{
  name={\ensuremath{\Delta_{\ell}(T)}},
  description={The characteristic element corresponding to the constant \(\symbf{Z}_p\)-tower over the double supersingular \(\ell\)-isogeny graph \(X^{(r,\ell)}\).}
  }
  \newglossaryentry{characteristicfactor}{
  name={\ensuremath{\Delta_{\ell,i}(T)}},
  description={The factor of \(\Delta_{\ell}(T)\) corresponding to a newform \(f_i \in \symscr{B}_2(r)\).}
  }
  \newglossaryentry{fouriercoeff}{
  name={\ensuremath{a_{\ell}(f)}},
  description={The \(\ell\)-th Fourier coefficient of a modular form \(f\).}
  }
  \newglossaryentry{r}{
  name={\ensuremath{r}},
  description={A prime number greater than \(5\). Below section \ref{sec_res}, it is assumed that \(r\equiv 1\pmod{12}\).}
  }
  \newglossaryentry{ell}{
  name={\ensuremath{\ell}},
  description={A prime number distinct from a prime number \(r\).}
  }
  \newglossaryentry{p}{
  name={\ensuremath{p}},
  description={A prime number. In section \ref{sec_res}, it is assumed to be odd.}
  }
  \newglossaryentry{cyclo}{
  name={\ensuremath{\chi_p}},
  description={The \(p\)-adic cyclotomic character.}
  }
  \newglossaryentry{rhofp}{
  name={\ensuremath{\rho_{f,p}}},
  description={A Galois representation attached to a newform \(f\in S_2(\Gamma_0(r))\).}
  }
  \newglossaryentry{absQ}{
  name={\ensuremath{G_{\symbf{Q}}}},
  description={The absolute Galois group of \(\symbf{Q}\).}
  }
  \newglossaryentry{kappa}{
  name={\ensuremath{\kappa(X)}},
  description={The complexity of a graph \(X\).}
  }
  \newglossaryentry{frakLC}{
  name={\ensuremath{\symfrak{L}(C)}},
  description={The set corresponding to a subset \(C \subset \Gal(K/\symbf{Q})\) stable under conjugations.}
  }
  \newglossaryentry{[n]}{
  name={\ensuremath{[n]}},
  description={The endomorphism multiplying points by an integer \(n\).}
  }
\begin{document}
\maketitle
\abstract{
A graph-theoretic analogue of Iwasawa theory, initiated by Gonet and Valli\`eres, has attracted considerable interest in the study of Iwasawa invariants.
On the other hand, for a pair of prime numbers \((r,\ell)\), one obtains a graph, called the supersingular \(\ell\)-isogeny graph (SIG), whose adjacency matrix has eigenvalues given by the \(\ell\)-th Fourier coefficients of the weight 2 Eisenstein series and newforms of level \(r\).
In this paper, we fix prime numbers \(r\) and \(p\), and let \(\ell\) vary over infinitely many primes.
We then investigate the distribution of the Iwasawa \(\lambda\)-invariants of the constant \(\symbf{Z}_p\)-towers over the SIGs, thereby revealing connections among graph theory, Iwasawa theory, elliptic curves, and the Galois representations attached to newforms.
At the end of this paper, we propose a conjecture concerning the Galois orbits of newforms.
}
\tableofcontents
\section{Introduction}
Theorem \ref{thm_5} below is a graph-theoretic analogue of the celebrated Iwasawa's class number formula in number theory \cite{Iwasawa}, and it was proved independently by Gonet \cite{Gon} and by McGown and Valli\`eres \cite{MV_24} in recent years.
The integers \(\mu\), \(\lambda\), and \(\nu\) in this theorem are called \emph{Iwasawa invariants}, and they are of particular interest.
\begin{theorem}[Gonet {\cite{Gon}}, McGown and Valli\`eres {\cite{MV_24}}]\label{thm_5}
Let \(p\) be a prime number.
Let 
\[
X \leftarrow X_1 \leftarrow X_2 \leftarrow \cdots \leftarrow X_n \leftarrow \cdots
\]
be a compatible system of \(\symbf{Z}/p^n\symbf{Z}\)-covers of finite connected graphs, which is called a \emph{\(\symbf{Z}_p\)-tower over \(X\)}.
Let \(\kappa(X_n)\) denote the \emph{complexity of \(X_n\)}, i.e., the number of spanning trees of \(X_n\).
Then there exist unique non-negative integers \(\mu, \lambda\), and an integer \(\nu\) such that, for all sufficiently large numbers \(n\), we have
\[
\ord_p\big(\kappa(X_n)\big) = \mu p^n +\lambda n + \nu,
\]
where \(\ord_p\big(\kappa(X_n)\big)\) denotes the exponent of \(p\) in \(\kappa(X_n)\).
\end{theorem}
In \cite{DLRV}, Dion, Lei, Ray, and Valli\`eres investigated the distributions of Iwasawa invariants, and in particular, they showed that Iwasawa \(\lambda\)-invariants are always odd if \(p \neq 2\).
Kataoka \cite{Kataoka} proved that for every pair \((\mu,\lambda)\) with \(\lambda\) odd, there exists a \(\symbf{Z}_p\)-tower of graphs whose Iwasawa invariants are \((\mu,\lambda)\).
On the other hand, Lei and M\"uller \cite{LM_a}, \cite{LM_b}, \cite{LM_c}, \cite{LM_d} study \(\symbf{Z}_p\)-towers of certain directed graphs called isogeny graphs, in particular \emph{supersingular \(\ell\)-isogeny graphs}, abbreviated as \(\ell\)-SIGs.
Isogeny graphs may be regarded as graph-theoretic analogues of modular curves over finite fields, and \(\ell\)-SIGs are graph-theoretic analogues of the supersingular loci of modular curves.
Indeed, the vertices of an \(\ell\)-SIG are isomorphism classes of supersingular elliptic curves over the algebraic closure \(\wideoverbar{\symbf{F}}_r\) of a finite field of characteristic \(r\ (\neq \ell)\) and the edges correspond to \(\ell\)-isogenies.
In particular, in \cite{LM_c}, they study \(\symbf{Z}_p\)-towers consisting of supersingular connected components of isogeny graphs.

In this paper, we treat SIGs of level 1.
The history of studies on such SIGs is long, and it is known that the eigenvalues of the adjacency matrix of an \(\ell\)-SIG of level 1 coincide with the Fourier coefficients of weight-2 Eisenstein series and newforms \cite{Eic_a}, \cite{Eic_b}, \cite{Gro}.
Each SIG is determined by the choice of a pair of prime numbers \(r,\ell\), and these \(f_i\) are determined only by the choice of \(r\).
By forgetting the direction of supersingular isogeny graphs, we can consider a graph, which we call the \emph{double supersingular isogeny graph}.
Given a double supersingular isogeny graph \(X\), a special \(\symbf{Z}_p\)-tower, called a constant \(\symbf{Z}_p\)-tower,
\[
X \leftarrow X_1 \leftarrow X_2 \leftarrow \cdots \leftarrow X_n \leftarrow \cdots
\]
of graphs is uniquely defined up to graph isomorphism.
Hence this tower is uniquely determined by a triple \((r,p,\ell)\).
The \(\symbf{Z}_p\)-tower of SIGs induced by their full level structures ``is a disjoint union of copies of a constant \(\symbf{Z}_p\)-tower.''
Such an interpretation is studied by Lei-M\"{u}ller \cite{LM_c}, and in Appendix A, we slightly modify their argument to agree with our situation.
In this paper, by fixing prime numbers \(r\) and \(p\), and letting \(\ell\) vary over infinitely many prime numbers with certain conditions, we investigate the distribution of the Iwasawa \(\lambda\)-invariants associated with \(\ell\).
Our main result is the following.
\begin{maintheorem}[Distribution theorem = (Main Theorem \ref{thm_2})]\label{thm_6}
Let \(\lambda_{\ell}\) denote the Iwasawa \(\lambda\)-invariant of the constant \(\symbf{Z}_p\)-tower over the double supersingular \(\ell\)-isogeny graph.
For the set \(\symscr{B}_2(r)\) of the weight 2 newforms of level \(r\), consider the newform orbits (orbits of newforms under the Galois conjugation) \(\symscr{B}_2(r)=O(1)\sqcup O(2)\sqcup\cdots\sqcup O(s)\).
Then, for all but finitely many \(p\), the following statement holds.
For an arbitrary subset \(W \subset \{1,\dotsc,s\}\),
the set
\[
\symscr{L}_W\coloneq \bigg\{\,\ell:
\text{
\(\ell\) is a prime number distinct from \(r\), \(\ell\equiv 1 \pmod{p}\), \(\lambda_{\ell}=1+2\sum_{k\smallin W}|O(k)|\)
}
\,\bigg\}
\]
has positive lower density.
Here, \textup{``\emph{has positive lower density}''} means that the limit inferior
\[
\liminf_{n\to\infty} \frac{|\{\,\ell\in \symscr{L}_W:\ell \leqq n\,\}|}{|\{\,\ell:\text{\(\ell\) is a prime number \(\leqq n\)}\,\}|}
\]
takes a positive value.
\end{maintheorem}
The size \(|O(k)|\) of a newform orbit is equal to the degree of the extension \([L_k:\symbf{Q}]\), where \(L_k\) is the field generated by the Fourier coefficients of a newform \(f_k \in O(k)\), known as the \emph{Hecke field} of \(f_k\).
In a database of newforms LMFDB \cite{LMFDB}, we can find the degree \([L_k:\symbf{Q}]\) and hence the size \(|O(k)|\) of the newform orbits.
Also note that we can give odd numbers of the form \(1+2\sum_{k\smallin W}|O(k)|\) if we fix a prime number \(r \equiv 1 \pmod{12}\).
Hence, we rephrase Main Theorem \ref{thm_6} as follows: 
for a prime number \(r\) with \(r\equiv 1 \pmod{12}\), let \(s\) be the number of the newform orbits, and consider an arbitrary subset \(W \subset \{1,\dotsc,s\}\).
Then, for all but finitely many odd prime numbers \(p\), the set of prime numbers \(\ell\) such that \(\lambda_{\ell} = 1+2\sum_{k\smallin W}|O(k)|\) has positive lower density.
In Table \ref{tab_1} at \S 5.2, we list the prime numbers \(r \equiv 1 \pmod{12}\) less than 2000 and odd numbers of the form \(1+2\sum_{k\smallin W}|O(k)|\).

Our proof of Main Theorem 1.2 strongly relies on the Chebotarev density theorem.
To apply this theorem, we have to control the \(p\)-adic behaviors of the Fourier coefficients \(a_{\ell}(f_1),\dotsc,a_{\ell}(f_s)\) of pairwise non-Galois-conjugate newforms \(f_1,\dotsc,f_s\) as \(\ell\) varies over prime numbers.
Big image theorem for the product of representations for newforms and a cyclotomic character in Lei, Loeffler, and Zerbes \cite{LLZ}, Loeffler \cite{Loe}---which is a generalization of one in Momose \cite{Momose} and Ribet \cite{Ribet_85}---plays a crucial role in controlling such behaviors.

Finally, we also remark that, in a closely related setting, Munier and Shnidman \cite{MS_23} computed the distribution of the \(p\)-Sylow subgroup of the Jacobian of \(\ell\)-SIGs as \(\ell\to\infty\), which may have some relation with our study.
It would be interesting to clarify the precise relation between their distribution results
and the Iwasawa-theoretic viewpoint developed in the present paper.
\subsection*{Organization of this paper}
In Section \ref{prel}, we review basic facts about Serre's formalism for graphs, Iwasawa-type formulas for graphs, and elliptic curves.
In Section \ref{SI}, we describe some basic facts about supersingular isogeny graphs, which are the main object of this paper.
In Section \ref{big}, we recall results of Loeffler et al.\ on the images of the finite product of representations attached to newforms, which are the key ingredients in the proof of our main results.
In Section \ref{sec_res}, we state and prove our main results.
In Appendix A, we provide an interpretation of our constant \(\symbf{Z}_p\)-tower in terms of a \(\symbf{Z}_p\)-tower of SIGs with full level structure, as studied in the papers of Lei and M\"uller \cite{LM_a}, \cite{LM_c}.

\subsection*{Acknowledgements}
The authors are grateful to Takenori Kataoka for reading an earlier version of this paper carefully, pointing out mathematical mistakes including an inaccuracy of the statement of Main Theorem \ref{thm_2}, and giving them a lot of insightful comments.
The authors also would like to thank Takashi Hara, Rikuto Ito, Shinichi Kobayashi, Antonio Lei, Shotaro Shirai, Kennichi Sugiyama, and ChatGPT for informative comments.
The first author was supported by WISE program (MEXT) at Kyushu University.
The second and third authors were partially supported by JST SPRING, Grant number JPMJSP2125.
The fourth author was supported by JSPS KAKENHI, Grant Number 26KJ0138.
\section*{Notations}
\(\symbf{F}_q\), \(\symbf{Z}\), \(\symbf{Q}\), \(\symbf{R}\), \(\symbf{C}\) denote respectively the finite field of order \(q\), the ring of rational integers, the field of rational numbers, the field of real numbers, and the field of complex numbers.
\printglossary[title=Glossary of symbols]

%%%%%%%%%%%%%%%%%%%%%%%%%%%%%%%%%%%%%%%%
\section{Preliminaries}\label{prel}
In this section, we briefly review foundations of graphs formalized by Serre, the theory of Iwasawa-type formulas for \(\symbf{Z}_p\)-towers of graphs, and supersingular elliptic curves, which are the main objects in this paper.
\subsection*{\starredbullet\,Graphs}
First, we recall some notions of graphs.
For more detail, see Chapter I of \cite{Serre} for graphs, and see \cite{DV}, \cite{Gon}, \cite{KM}, or the authors' papers \cite{AMT}, \cite{MT} for Iwasawa-type formulas.
\begin{definitions}
\begin{enumerate}[label=(\arabic*)]
\item 
A \emph{directed graph} \(X\) is a quadruplet \((\symbf{V}(X),\symbf{E}(X),o,t)\) where \gls{vertices} and \gls{edges} are sets, and \(o\) and \(t\) are maps \(\symbf{E}(X)\to \symbf{V}(X)\).
The elements of \(\symbf{V}(X)\), \(\symbf{E}(X)\) are called \emph{vertices} and \emph{edges} of \(X\) respectively.
For each edge \(e\in \symbf{E}(X)\), \(o(e)\) and \(t(e)\) are called \emph{origin} and \emph{terminus} of \(e\) respectively.
\item 
A \emph{graph} (in the sense of \cite{Serre}) is a directed graph \(X\) together with a map \((\cdot)\,\bar{}\,\colon \symbf{E}(X)\to \symbf{E}(X)\) such that \(\bar{e}\neq e\), \((\bar{e})\,\bar{}\,=e\), \(o(e)=t(\bar{e})\) for every \(e\in \symbf{E}(X)\), and we call \(\bar{e}\) the \emph{inverse edge} of \(e\).
In this situation, one has a decomposition \(\symbf{E}(X)=S_1\sqcup S_2\) such that \(\wideoverbar{S}_1=S_2\), and forgetting the directions of \(S_1\) (or \(S_2\)) induces the undirected graph \(X_{\mathrm{undir}}\) with the same vertices and incidence relations of \(X\).
In this meaning, the decomposition describes an \emph{orientation} of \(X_{\mathrm{undir}}\) by \(S_1\) and the \emph{inverse orientation} by \(S_2\).
Then the usual notion of undirected graphs, e.g., \emph{finiteness}, \emph{connectedness}, \emph{closed paths}, and \emph{subgraphs} are defined.
%Henceforth we simply call symmetric directed graph by \emph{graph}.
\item
The \emph{adjacency matrix} \gls{adjacencymatrix} of a directed graph \(X\), including Serre formalism graph, is a square matrix labeled by the set of vertices \(\symbf{V}(X)\) whose \((u,v)\)-entry is given by
\begin{equation}
A(X)_{u,v}\coloneq \big|\{\,e \in \symbf{E}(X): o(e)=u,\ t(e)=v\,\}\big|,
\end{equation}
and the \emph{degree matrix \gls{degreematrix}} is a diagonal matrix labeled by \(\symbf{V}(X)\) whose \((v,v)\)-entry is given by
\begin{equation}
D(X)_{v,v}\coloneq \big|\{\,e \in \symbf{E}(X): o(e)=v\,\}\big|.
\end{equation}
\item A \emph{spanning tree} \(T\) of a finite graph \(X\) is a connected subgraph of \(X\) such that \(\symbf{V}(T)=\symbf{V}(X)\) and \(T\) has no cycles.
The \emph{complexity} \gls{kappa} of a graph \(X\) is defined to be the number of spanning trees of \(X\).
Note that if \(X\) is disconnected, then \(\kappa(X)=0\).
\end{enumerate}
\end{definitions}
\begin{remark}
One of the differences of graphs and undirected graphs is that diagonal entries (numbers of loops on each vertex) of the adjacency matrix of a graph are always even numbers, while the ones of an undirected graph may be odd numbers.
\end{remark}
Next we recall morphisms and Galois covers of graphs.
\begin{definitions}
Let \(X\) and \(Y\) be \emph{finite connected} graphs.
\begin{enumerate}[label=(\arabic*)]
\item A \emph{morphism} \(f\) from \(Y\) to \(X\) is a pair of maps \(f_{\symbf{V}}\colon \symbf{V}(Y)\to \symbf{V}(X)\) and \(f_{\symbf{E}}\colon \symbf{E}(Y)\to\symbf{E}(X)\) such that \(f_{\symbf{V}}(o(e))=o(f_{\symbf{E}}(e))\), \(f_{\symbf{V}}(t(e))=t(f_{\symbf{E}}(e))\), \(f_{\symbf{E}}(\bar{e})=\big(f_{\symbf{E}}(e)\big)\,\bar{}\) for every \(e\in \symbf{E}(Y)\).
An \emph{automorphism} \(f\colon Y\similarrightarrow Y\) is a morphism on a graph \(Y\) whose maps \(f_{\symbf{V}}\) and \(f_{\symbf{E}}\) are bijective.
Of course, for every automorphism \(f\), one can construct the inverse \(f^{-1}\) naturally.
Put \(\Aut(Y)\coloneq \{\,\text{automorphisms of \(Y\)}\,\}\).
\item A morphism \(\pi\colon Y\to X\) is called an (unramified) \emph{cover} if the maps \(\pi_{\symbf{V}}\) and \(\pi_{\symbf{E}}\) are surjective, and for each \(v\in \symbf{V}(Y)\), \(\pi_{\symbf{E}}\) induces a bijective map
\[
\{\,e \in \symbf{E}(Y):o(e)=v\,\} \xrightarrow{\mathrm{bij.}} \{\,e \in \symbf{E}(X):o(e)=\pi_{\symbf{V}}(v)\,\}.
\]
As usual, we write \(\pi\colon Y\to X\) as \(Y/X\) if \(\pi\) is a cover.
Put \(\Aut(Y/X)\coloneq \{\,f\in \Aut(Y): \pi \circ f = \pi\,\}\), and call this the \emph{group of deck transformations} of \(Y/X\).
Then, a cover is said to be \emph{Galois} if for each \(v\in \symbf{V}(X)\), the induced action \(\Aut(Y/X)\curvearrowright \pi_V^{-1}(v)\) is transitive.
For a Galois cover, we write \(\Gal(Y/X)\coloneq \Aut(Y/X)\), and call this the \emph{Galois group} of \(Y/X\).
It is known that for a Galois cover \(Y/X\), there is a correspondence between the subgroups of \(\Gal(Y/X)\) and the set of intermediate graphs of \(Y/X\); see Part IV in \cite{Terras} for more details.
\end{enumerate}
\end{definitions}
One of the most important facts is that one can construct Galois covers using voltage assignments.
\begin{definition-proposition}\label{def-prop_1}
Consider a connected graph \(X\) and a finite group \(G\).
A \emph{voltage assignment} is a map \(\alpha\colon \symbf{E}(X)\to G\) such that \(
\alpha(\bar{e})=\alpha(e)^{-1}\) for every edge \(e\in \symbf{E}(X)\).
Practically, one can obtain a voltage assignment as follows.
Choose an orientation \(\symbf{E}(X)=S_1\sqcup S_2\) (so \(\wideoverbar{S}_1=S_2\)) and an arbitrary map \(\alpha_1\colon S_1\to G\), then setting \(\alpha_2\colon S_2\to G;\ e\mapsto \alpha_1(e)^{-1}\), \(\alpha_1\) and \(\alpha_2\) define a voltage assignment \(\alpha\).
For a voltage assignment \(\alpha\), put \(V\big(X(\alpha)\big)\coloneq \symbf{V}(X)\times G\), \(\symbf{E}\big(X(\alpha)\big)\coloneq \symbf{E}(X)\times G\), and for each \((e,\sigma)\in \symbf{E}\big(X(\alpha)\big)\), define the directions by
\[
o\big((e,\sigma)\big)\coloneq (o(e),\sigma),\quad t\big((e,\sigma)\big)\coloneq (t(e),\sigma\alpha(e)),\quad (e,\sigma)\,\bar{}\coloneq \big(\bar{e},\sigma\alpha(e)\big).
\]
Then one has the following.
\begin{enumerate}[label=\((\arabic*)\)]
\item \(X(\alpha)\) is a finite graph, and natural projections \(\symbf{V}\big(X(\alpha)\big)\twoheadrightarrow \symbf{V}(X)\), \(\symbf{E}\big(X(\alpha)\big)\twoheadrightarrow \symbf{E}(X)\) give a cover \(X(\alpha)/X\).
We call \(X(\alpha)\) the \emph{derived graph associated to \(\alpha\)}.
\end{enumerate}
In the following, we assume that \(X\) is connected.
\begin{enumerate}[label=\((\arabic*)\)]
\setcounter{enumi}{1}
\item \(X(\alpha)\) is connected if and only if the set 
\[
\{\,\alpha(e_1)\cdots\alpha(e_l):\text{\(l \geqq 1\), \((v_1\xrightarrow{e_1}v_2\to \cdots \to v_{l}\xrightarrow{e_l} v_1)\) is a closed path of length \(l\)}\,\}
\]
generates the group \(G\), where a \emph{cycle} is a closed path with vertices \(v_1,\dotsc,v_l\).
\item If \(X(\alpha)\) is connected, then \(X(\alpha)/X\) is a Galois cover with Galois group \(G\).
\end{enumerate}
\end{definition-proposition}
\subsection*{\starredbullet\,Iwasawa-type formulas}
We now describe \(\symbf{Z}_p\)-towers using voltage assignments.
Let \gls{p} be a prime number, \(\symbf{Q}_p\) be the field of \(p\)-adic numbers, and \(\symbf{Z}_p\) be its ring of integers.
Even though \(\symbf{Z}_p\) is an infinite additive group, consider a voltage assignment \(\alpha\colon \symbf{E}(X) \to \symbf{Z}_p\) with the same conditions as in the finite group case.
Then, for each positive integer \(n\), composing \(\alpha\) with the natural projection \(\symbf{Z}_p\twoheadrightarrow \symbf{Z}_p/p^n\symbf{Z}_p\) induces a genuine voltage assignment \(\alpha_n\colon \symbf{E}(X) \to \symbf{Z}_p/p^n\symbf{Z}_p (\cong \symbf{Z}/p^n\symbf{Z})\), and thus we obtain a sequence of morphisms 
\[
X\leftarrow X_1 \leftarrow \cdots \leftarrow X_n\leftarrow \cdots, \qquad X_n \coloneq X(\alpha_n).
\]
If some cycles of \(X\) are assigned to units of \(\symbf{Z}_p\), then every \(X(\alpha_n)\) is connected, and thus every morphism in this sequence is a Galois cover.
In such a situation, we call this sequence a \emph{\(\symbf{Z}_p\)-tower over \(X\)}.
The graph-theoretic Iwasawa-type formula for a \(\symbf{Z}_p\)-tower is the following.
\begin{theorem}[Gonet {\cite[Theorem 1.1]{Gon}}, McGown and Valli\`eres {\cite[Theorem 6.1]{MV_24}}]
Let \(X \leftarrow X_1\leftarrow \cdots \leftarrow X_n\leftarrow\cdots\) be a \(\symbf{Z}_p\)-tower over \(X\).
Then there exist unique integers \(\mu\), \(\lambda\), \(\nu\) such that for sufficiently large numbers \(n\), we have the formula
\begin{equation}\label{eq_5}
\ord_p\big(\kappa(X_n)\big) = \mu p^n + \lambda n +\nu,
\end{equation}
where \(\ord_p\big(\kappa(X_n)\big)\) stands for the exponent of \(p\) in the prime factorization of \(\kappa(X_n)\).
The integers \(\mu\), \(\lambda\), \(\nu\) are called \emph{Iwasawa invariants}.
\end{theorem}
The Iwasawa \(\mu\)-invariant and Iwasawa \(\lambda\)-invariant can be calculated from data of a voltage assignment.
\begin{definition}
Let \(L\) be a number field with ring of integers \(\symscr{O}_L\), and choose a prime ideal \(\symfrak{P}\) of \(\symscr{O}_L\) lying above \(p\).
Let \(L_{\symfrak{P}}\) be a completion of \(L\) by \(\symfrak{P}\), and let \(\symscr{O}_{\symfrak{P}}\) be its valuation ring with uniformizer \(\varpi\).
Consider the ring of formal power series \(\Lambda\coloneq \symscr{O}_{\symfrak{P}}\lBrack T\rBrack\), which is known as the \emph{Iwasawa algebra}.
For every non-zero series \(F\in \Lambda\otimes_{\symscr{O}_{\symfrak{P}}}L_{\symfrak{P}}\), there are a unique integer \(N(F)\) and a series \(F_0 \in \Lambda\) such that \(F = \varpi^{N(F)}F_0\) and \(\varpi\not\mid F_0\).
Writing \(e(L_{\symfrak{P}}/\symbf{Q}_p)\) for the ramification index of \(L_{\symfrak{P}}/\symbf{Q}_p\), we call the rational number \(\mu(F)\coloneq N(F)/e(L_{\symfrak{P}}/\symbf{Q}_p)\) the \emph{Iwasawa \(\mu\)-invariant of} \(F\).
Next consider the image \(\wideoverbar{F}_0 \in \Lambda/\symfrak{P}\Lambda \cong (\symscr{O}_{\symfrak{P}}/\symfrak{P}\symscr{O}_{\symfrak{P}})\lBrack T\rBrack\).
We call the degree \(\lambda(F)\) of the lowest monomial in \(\wideoverbar{F}_0\) the \emph{Iwasawa \(\lambda\)-invariant of} \(F\).
\end{definition}
\begin{proposition}[McGown and Valli\`eres {\cite[Theorem 6.1]{MV_24}}, Adachi, Mizuno, and Tateno {\cite[Theorem 3.9]{AMT}}]\label{prop_3}
Let \(X \leftarrow X_1\leftarrow \cdots \leftarrow X_n\leftarrow\cdots\) be a \(\symbf{Z}_p\)-tower with a voltage assignment \(\alpha\colon \symbf{E}(X)\to\symbf{Z}_p\).
Put \(\symbf{E}(X)_{u,v}\coloneq\{\,e\in \symbf{E}(X):o(e)=u,\ t(e)=v\,\}\) for each \(u,v\in \symbf{V}(X)\).
Consider the series
\begin{equation}\label{eq_6}
\Delta(T)\coloneq \det \raisebox{-1.00ex}{\bigg(}D(X)-\raisebox{-0.75ex}{\bigg(}\sum_{e\smallin \symbf{E}(X)_{u,v}} (1+T)^{\alpha(e)}\raisebox{-0.75ex}{\(\bigg)_{u,v\smallin \symbf{V}(X)}\)}\raisebox{-1.00ex}{\bigg)} \in \Lambda,
\end{equation}
and call this series the \emph{characteristic element} of the \(\symbf{Z}_p\)-tower.
Then, we have \(\Delta(T)\neq 0\), and Iwasawa invariants in \eqref{eq_5} are given by those of the characteristic element \(\Delta\) as
\(
\mu= \mu(\Delta),\ \lambda = \lambda(\Delta)-1.
\)
\end{proposition}
\subsection*{\starredbullet\,Elliptic curves}
In this subsection, we recall some basic facts of elliptic curves over a finite field; see \cite[Chapter IV]{Har}, \cite{Silv} for an elementary treatment of elliptic curves.
\begin{definitions}
Let \gls{r} be a prime number with \(r\geqq 5\) (we mostly take \(r\) with \(r\equiv 1 \pmod{12}\)), let \(\symbf{F}_r\) denote the finite field of order \(r\), and fix an algebraic closure \(\symbf{F}_r\hookrightarrow \wideoverbar{\symbf{F}}_r\).
\begin{enumerate}[label=(\arabic*)]
\item An \emph{elliptic curve} \(E\) stands for a non-singular projective curve \(E\) over \(\wideoverbar{\symbf{F}}_r\), of genus 1, with a fixed point \(0_E\).
The point \(0_E\) induces a group structure on the (closed) points \(E\raisebox{0.25ex}{\big(}\wideoverbar{\symbf{F}}_r\raisebox{0.25ex}{\big)}\) whose zero element is \(0_E\) and an embedding into the projective plane \(\symbf{P}^2\) as a cubic curve.
Each elliptic curve has a \emph{\(j\)-invariant}, which is a \(\wideoverbar{\symbf{F}}_r\)-valued invariant.
Since \(r\geqq 5\), taking \(j\)-invariants gives  a one-to-one correspondence between isomorphism classes of elliptic curves over \(\wideoverbar{\symbf{F}}_r\) and the elements of \(\wideoverbar{\symbf{F}}_r\), so \(\wideoverbar{\symbf{F}}_r\) can be seen as a moduli space of isomorphism classes of elliptic curves.
\item A \emph{supersingular} elliptic curve \(E\) is an elliptic curve whose group of \(r\)-torsion points \(E\raisebox{0.25ex}{\big(}\wideoverbar{\symbf{F}}_r\raisebox{0.25ex}{\big)}[r]\) is trivial.
A supersingular curve \(E\) can be defined over \(\symbf{F}_{r^2}\), that is, \(E\) is the base extension of a curve in \(\symbf{P}^2\) over \(\symbf{F}_{r^2}\) and \(0_E\) is an \(\symbf{F}_{r^2}\)-rational point, in other words, we can choose a defining polynomial of \(E\) whose coefficients are in \(\symbf{F}_{r^2}\) and the coordinates of \(0_E\) are in \(\symbf{F}_{r^2}\).
The \(j\)-invariant of a supersingular elliptic curve is called the \emph{supersingular \(j\)-invariant}.
\item An \emph{isogeny} \(g\) is a non-constant morphism \(g\colon E_1\to E_2\) between elliptic curves which sends \(0_{E_1}\mapsto 0_{E_2}\).
As isogenies are non-constant morphisms between complete curves, they are surjective and finite, and induce group homomorphisms \(E_1\raisebox{0.25ex}{\big(}\wideoverbar{\symbf{F}}_r\raisebox{0.25ex}{\big)}\to E_2\raisebox{0.25ex}{\big(}\wideoverbar{\symbf{F}}_r\raisebox{0.25ex}{\big)}\).
For a positive number \(n\), an \emph{\(n\)-isogeny} is an isogeny \(g\) whose kernel is of order \(n\), or equivalently, it is a morphism of degree \(n\).
Every \(n\)-isogeny \(g\colon E_1\to E_2\) has a unique \(n\)-isogeny \(\hat{g}\colon E_2\to E_1\), called the \emph{dual isogeny} of \(g\), such that the composite morphism \(\hat{g}\circ g\colon E_1\to E_1\) is the endomorphism \(\gls{[n]}\) multiplying points by \(n\).
It is important that for every isogeny \(g\), the double dual of \(g\) is itself, that is, \((\hat{g})\,\hat{}\,=g\) holds.
\end{enumerate}
\end{definitions}
According to \cite[\S III.4]{Silv}, for a prime number \(\ell\neq r\), a (separable) \(\ell\)-isogeny \(E_1\to E_2\) is determined by a non-trivial subgroup of \(E_1\raisebox{0.25ex}{\big(}\wideoverbar{\symbf{F}}_r\raisebox{0.25ex}{\big)}[\ell]\) up to automorphisms of \(E_2\), and conversely, every non-trivial subgroup \(\Phi\) of \(\ell\)-torsion points induces a \(\ell\)-isogeny from \(E_1\) to a \emph{quotient elliptic curve} \(E_1/\Phi\).
If \(E_1\) is supersingular, and separably isogenous to \(E_2\), then \(E_2\) is also supersingular.
We can confirm this fact group theoretically as follows.
\begin{proof}
If \(E_2\) had a non-zero \(r\)-torsion point, then \(E_1\) has a point \(P\) with \(P\not\in \ker g\) and \(r\cdot P \in \ker g\) for some \(\ell\)-isogeny \(g\colon E_1\to E_2\).
Since \(\ker g\) is of order \(\ell\), \(r\ell\cdot P = 0_{E_1}\), and hence \(\ell\cdot P = 0_{E_1}\) since \(E_1\) is supersingular.
However, as \(r\) and \(\ell\) are distinct primes, we can find integers \(a\), \(b\) such that \(P = (ar+b\ell)P \in \ker g\), which is a contradiction.
\end{proof}
%%%%%%%%%%%%%%%%%%%%%%%%%%%%%%%%%%%%
\section{Supersingular isogeny graphs}\label{SI}
In this section, we describe some basic facts about supersingular isogeny graphs.
Goren and Love \cite{GL} is a useful reference for supersingular isogeny graphs and several related topics.

Fix a complete set \(\{E_0, E_1,\dotsc,E_h\}\) of representatives of the isomorphism classes of supersingular elliptic curves over \(\wideoverbar{\symbf{F}}_r\).
\begin{definitions}\label{def_1}
Let \gls{ell} be a prime number distinct from \(r\).
\begin{enumerate}[label=(\arabic*)]
\item The \emph{supersingular \(\ell\)-isogeny graph \gls{supersingularisogeny} over \(\wideoverbar{\symbf{F}}_r\)} is the directed graph whose set of vertices is given by
\begin{equation}
\symbf{V}(\SI(r,\ell))\coloneq \big\{E_0,E_1,\dotsc,E_h\big\},
\end{equation}
and for two vertices \(E,E'\), the set of directed edges from \(E\) to \(E'\) is given by
\begin{equation}
\symbf{E}(\SI(r,\ell))_{E,E'}\coloneq \big\{\,\text{\(\ell\)-isogenies \(E\to E'\) up to automorphisms of \(E'\)}\,\big\}.
\end{equation}
For simplicity, we abbreviate supersingular \(\ell\)-isogeny graph to \emph{\(\ell\)-SIG}.
\item The \emph{double supersingular \(\ell\)-isogeny graph}  (abbreviated to \emph{\(\ell\)-DSIG}) \gls{doublesupersingularisogeny} is the graph constructed as follows.
Its set of vertices \(\symbf{V}(X^{(r,\ell)}) \coloneq \symbf{V}(\SI(r,\ell))\).
Its set of edges is given, as a set, by the disjoint union \(\symbf{E}(X^{(r,\ell)})\coloneq \symbf{E}(\SI(r,\ell))\coprod \symbf{E}(\SI(r,\ell))\).
We embed \(\symbf{E}(\SI(r,\ell))\) into \(\symbf{E}(X^{(r,\ell)})\) as the first summand, and give edge directions in the same way as \(\SI(r,\ell)\).
Then we consider the edges in the second summand as the corresponding inverse edges.
In this meaning, we can write the second summand as \(\wideoverbar{\symbf{E}(\SI(r,\ell))}\).
For an isogeny \(g\) and its dual isogeny \(\hat{g}\), one should not confuse \((\hat{g})\,\bar{}\) with \(g\).
\end{enumerate}
\end{definitions}
We summarize the basic properties of SIGs.

\noindent\emph{Number of vertices}\hspace{1em}
From \cite[V.4.1c]{Silv}, the number of vertices \(h+1\) is nearly equal to \(\lfloor r/12 \rfloor\), where \(\lfloor\cdot\rfloor\) is the floor function.
In particular, \(h+1= \lfloor r/12 \rfloor\) if \(r \equiv 1 \pmod{12}\).

\noindent\emph{Outdegree}\hspace{1em}
The \emph{outdegree at a vertex \(E \in \symbf{V}(\SI(r,\ell))\)} is the number of edges of \(\SI(r,\ell)\) whose origin is \(E\).
As stated in preliminaries, we know that the number of edges starting from \(E\) is the number of non-trivial subgroups of \(\ell\)-torsion points.
By \cite[Corollary 6.4]{Silv}, \(E\raisebox{0.25ex}{\big(}\wideoverbar{\symbf{F}}_r\raisebox{0.25ex}{\big)}[\ell] \cong (\symbf{Z}/\ell\symbf{Z})\times(\symbf{Z}/\ell\symbf{Z})\), and hence the outdegree of every vertex of \(\SI(r,\ell)\) is \(\ell+1\).
In other words, \(D(\SI(r,\ell))=(\ell+1)I_{h+1}\) where \(I_{h+1}\) denotes the identity matrix of size \(h+1\).

\noindent\emph{Connectedness}\hspace{1em}
Since \(\symbf{E}(\SI(r,\ell))_{E,E'}\neq \emptyset\) implies \(\symbf{E}(\SI(r,\ell))_{E'\!,E}\neq \emptyset\), if there exists a path from \(E\) to \(E''\), then there exists a path from \(E''\) to \(E\).
This property allows us to consider the decomposition of \(\SI(r,\ell)\) into its connected components.
It is non-trivial but known that \(\SI(r,\ell)\) is connected, that is, \(\SI(r,\ell)\) has exactly one connected component, and so is  \(X^{(r,\ell)}\).
We will see the connectedness in corollary \ref{cor_1} below using the theory of modular forms.

\noindent\emph{Symmetry}\hspace{1em}
Since \((\hat{g})\,\hat{}\,=g\) for every isogeny \(g\), we have a bijective map
\begin{equation}\label{eq_9}
(\cdot)\,\hat{}\,\colon\{\,\text{\(\ell\)-isogenies from \(E\) to \(E'\)}\,\}\to \{\,\text{\(\ell\)-isogenies from \(E'\) to \(E\)}\,\}.
\end{equation}
The set of edges from \(E\) to \(E'\) is the set of equivalence classes of \(\ell\)-isogenies from \(E\) to \(E'\) up to automorphisms of \(E'\).
The table in \cite[III.10.1]{Silv} shows that the group of automorphisms of an elliptic curve \(E\) is \(\{[1],[-1]\}\) unless its \(j\)-invariant is 0 or 1728.
Moreover, if \(r \equiv 1\pmod{12}\), 0 and 1728 are not supersingular \(j\)-invariants, see \cite[V.4.4 and V.4.5]{Silv} for this fact.
Hence, if \(r\equiv 1 \pmod{12}\), the above bijective map \eqref{eq_9} induces a bijective map \((\cdot)\,\hat{}\,\colon\symbf{E}(\SI(r,\ell))_{E,E'} \to \symbf{E}(\SI(r,\ell))_{E'\!,E}\), which leads us to the following important property in our results.
\begin{proposition}\label{prop_2}
Let \(\wideoverbar{\SI}(r,\ell)\) be the directed graph defined by reversing the origin and terminus of each edge of \(\SI(r,\ell)\).
Then, if \(r\equiv 1\pmod{12}\), one has the following.
\begin{enumerate}[label=\((\arabic*)\)]
\item 
The adjacency matrix \(A(\SI(r,\ell))\) is symmetric,
\item \(A\raisebox{0.125ex}{\big(}\wideoverbar{\SI}(r,\ell)\raisebox{0.125ex}{\big)}=A(\SI(r,\ell))\), 
\item \(A(X^{(r,\ell)})=2A(\SI(r,\ell))\).
\end{enumerate}
\end{proposition}
\begin{proof}
For (1), by the above bijective map \((\cdot)\,\hat{}\,\colon\symbf{E}(\SI(r,\ell))_{E,E'} \to \symbf{E}(\SI(r,\ell))_{E'\!,E}\) implies that 
\[
A(\SI(r,\ell))_{E,E'} = |\symbf{E}(\SI(r,\ell))_{E,E'}| = |\symbf{E}(\SI(r,\ell))_{E'\!,E}| =  A(\SI(r,\ell))_{E'\!,E}
\]
for each \(E\), \(E'\).
This means that \(A(\SI(r,\ell))\) is a symmetric matrix.
Then (2) follows from the equality \(A(\wideoverbar{\SI}(r,\ell)) = A(\SI(r,\ell))^{\symsf{T}}\), and (3) follows from the equality \(A(X^{(r,\ell)})=A(\SI(r,\ell))+A(\wideoverbar{\SI}(r,\ell))\).
\end{proof}
\begin{example}\label{exa_1}
We demonstrate a calculation of a SIG for the case that \(r=37\) and \(\ell = 2\).
The open source software SageMath \cite{sage} is very helpful for computing on \(\symbf{F}_{37^2}\).
Since \(r\equiv 1 \pmod{12}\), the number \(h+1\) of vertices is \(3\), and from \cite[TABLE 6]{BK}, the supersingular \(j\)-invariants are the roots of the polynomial \((j-8)(j^2-6j-6)\), namely, \(j_0\coloneq 8\), \(j_1\coloneq 3+\sqrt{15}\), \(j_2\coloneq 3-\sqrt{15}\), where we fix a square root \(\sqrt{15}\) of 15 in \(\wideoverbar{\symbf{F}}_{37}\).
Consider elliptic curves \(E_0\), \(E_1\), \(E_2\) with \(j\)-invariant 8, \(3+\sqrt{15}\), \(3-\sqrt{15}\) respectively.
We can obtain an elliptic curve from a given \(j\)-invariant; see the proof of proposition 1.4 (c) of \cite[\S III.1]{Silv}.
After changing coordinates appropriately, we obtain an equation \(y^2=x^3+12x+13\) of \(E_0\).
For a point \(P=(x,y)\in E_0\), its inverse is given by \(-P=(x,-y)\), so 2-torsion points are points \((x,0)\) with \(x^3+12x+13=0\).
Hence we obtain the non-trivial 2-torsion points
\[
(-1,0),\qquad(19+\sqrt{15},0),\qquad(19-\sqrt{15},0).
\]
Using V\'elu's formulas \cite{Velu}, we can obtain a quotient elliptic curve for any given non-trivial subgroup;
for the 2-subgroups \(\langle (-1,0)\rangle\), \(\langle (19+\sqrt{15},0)\rangle\), \(\langle (19-\sqrt{15},0)\rangle\), corresponding quotient elliptic curves are given by 
\[
y^2=x^3+11x+7,\quad y^2=x^3+(10-15\sqrt{15})x-17\sqrt{15},\quad y^2=x^3+(10+15\sqrt{15})x+17\sqrt{15}
\]
whose \(j\)-invariants are \(j_0\), \(j_1\), \(j_2\) respectively.
Consider the second as a defining equation for \(E_1\) and the third as a defining equation for \(E_2\).
The non-trivial 2-torsion points of \(E_1\) are
\[
(-1-2\sqrt{15},0),\qquad (-6-17\sqrt{15},0),\qquad (7-18\sqrt{15},0),
\]
and corresponding quotient elliptic curves are
\begin{align}
 y^2=x^3+7x+18,\quad y^2&=x^3+(-4-3\sqrt{15})x+(22+23\sqrt{15}),\\
 y^2&=x^3+(-8-8\sqrt{15})x+(-3+2\sqrt{15})
\end{align}
whose \(j\)-invariants are \(j_0\), \(j_2\), \(j_2\) respectively.
The non-trivial 2-torsion points of \(E_2\) are
\[
(7+18\sqrt{15},0),\qquad (-6+17\sqrt{15},0),\qquad (-1+2\sqrt{15},0),
\]
and quotient elliptic curves are
\begin{align}
 y^2=x^3+7x+18,\quad y^2&=x^3+(-8+8\sqrt{15})x+(-3-2\sqrt{15}),\\
 y^2&=x^3+(-4+3\sqrt{15})x+(22+14\sqrt{15})
\end{align}
whose \(j\)-invariants are \(j_0\), \(j_1\), \(j_1\) respectively.
As a consequence, we can picture \(\SI(37,2)\) and \(X^{(37,2)}\) as in Figure \ref{fig_1},
\begin{figure}
\centering
\begin{minipage}{14em}
\centering
\begin{tikzpicture}[scale=1.8]
\node (E0) at (0,0){\(E_0\)};
\node (E0rr) at ($(E0)+(45:0.03)$){\phantom{\(E_0\)}};
\node (E0rl) at ($(E0)+(-135:0.03)$){\phantom{\(E_0\)}};
\node (E0ll) at ($(E0)+(135:0.03)$){\phantom{\(E_0\)}};
\node (E0lr) at ($(E0)+(-45:0.03)$){\phantom{\(E_0\)}};
\node (E1) at (-45:1){\(E_1\)};
\node (E1r) at ($(E1)+(45:0.03)$){\phantom{\(E_1\)}};
\node (E1l) at ($(E1)+(-135:0.03)$){\phantom{\(E_1\)}};
\node (E1u) at ($(E1)+(135:0.03)+(90:0.01)$){\phantom{\(E_1\)}};
\node (E1d) at ($(E1)+(-135:0.03)+(-90:0.01)$){\phantom{\(E_1\)}};
\node (E1uu) at ($(E1)+(135:0.03)+(90:0.07)$){\phantom{\(E_1\)}};
\node (E1dd) at ($(E1)+(-135:0.03)+(-90:0.07)$){\phantom{\(E_1\)}};
\node (E2) at (225:1){\(E_2\)};
\node (E2r) at ($(E2)+(-45:0.03)$){\phantom{\(E_2\)}};
\node (E2l) at ($(E2)+(135:0.03)$){\phantom{\(E_2\)}};
\node (E2u) at ($(E2)+(45:0.03)+(90:0.01)$){\phantom{\(E_2\)}};
\node (E2d) at ($(E2)+(-45:0.03)+(-90:0.01)$){\phantom{\(E_2\)}};
\node (E2uu) at ($(E2)+(45:0.03)+(90:0.07)$){\phantom{\(E_2\)}};
\node (E2dd) at ($(E2)+(-45:0.03)+(-90:0.07)$){\phantom{\(E_2\)}};
\draw[-stealth] (E0)to[loop,looseness=3.5,out=60,in=120](E0);

\draw[-stealth] (E0rr)to(E1r);
\draw[-stealth] (E1l)to(E0rl);

\draw[-stealth] (E0lr)to(E2r);
\draw[-stealth] (E2l)to(E0ll);

\draw[-stealth] (E1u)to(E2u);
\draw[-stealth] (E1d)to(E2d);

\draw[-stealth] (E2uu)to(E1uu);
\draw[-stealth] (E2dd)to(E1dd);
\end{tikzpicture}
\end{minipage}
\begin{minipage}{14em}
\centering
\begin{tikzpicture}[scale=1.8]
\node (E0) at (0,0){\(E_0\)};
\node (E0rr) at ($(E0)+(45:0.03)$){\phantom{\(E_0\)}};
\node (E0rl) at ($(E0)+(-135:0.03)$){\phantom{\(E_0\)}};
\node (E0ll) at ($(E0)+(135:0.03)$){\phantom{\(E_0\)}};
\node (E0lr) at ($(E0)+(-45:0.03)$){\phantom{\(E_0\)}};
\node (E1) at (-45:1){\(E_1\)};
\node (E1r) at ($(E1)+(45:0.03)$){\phantom{\(E_1\)}};
\node (E1l) at ($(E1)+(-135:0.03)$){\phantom{\(E_1\)}};
\node (E1u) at ($(E1)+(135:0.03)+(90:0.01)$){\phantom{\(E_1\)}};
\node (E1d) at ($(E1)+(-135:0.03)+(-90:0.01)$){\phantom{\(E_1\)}};
\node (E1uu) at ($(E1)+(135:0.03)+(90:0.07)$){\phantom{\(E_1\)}};
\node (E1dd) at ($(E1)+(-135:0.03)+(-90:0.07)$){\phantom{\(E_1\)}};
\node (E2) at (225:1){\(E_2\)};
\node (E2r) at ($(E2)+(-45:0.03)$){\phantom{\(E_2\)}};
\node (E2l) at ($(E2)+(135:0.03)$){\phantom{\(E_2\)}};
\node (E2u) at ($(E2)+(45:0.03)+(90:0.01)$){\phantom{\(E_2\)}};
\node (E2d) at ($(E2)+(-45:0.03)+(-90:0.01)$){\phantom{\(E_2\)}};
\node (E2uu) at ($(E2)+(45:0.03)+(90:0.07)$){\phantom{\(E_2\)}};
\node (E2dd) at ($(E2)+(-45:0.03)+(-90:0.07)$){\phantom{\(E_2\)}};
\draw (E0)to[loop,looseness=3.5,out=60,in=120](E0);

\draw (E0rr)to(E1r);
\draw (E1l)to(E0rl);

\draw (E0lr)to(E2r);
\draw (E2l)to(E0ll);

\draw (E1u)to(E2u);
\draw (E1d)to(E2d);

\draw (E2uu)to(E1uu);
\draw (E2dd)to(E1dd);
\end{tikzpicture}
\end{minipage}
\caption[From left, the directed graph \(\SI(37,2)\), and the graph \(X^{(37,2)}\)]
{From left, the directed graph \(\SI(37,2)\), and the graph \(X^{(37,2)}\).
For the picture of \(X^{(37,2)}\), we draw a single line segment for each pair of an edge and its inverse edge.
}
\label{fig_1}
\end{figure}
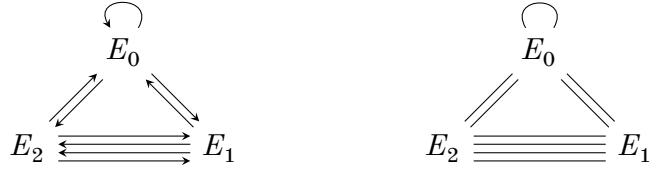
and we obtain the degree matrices and the adjacency matrices
\begin{alignat}{2}
D(\SI(37,2))&=
\begin{pmatrix}
3&0&0\\
0&3&0\\
0&0&3
\end{pmatrix}
,&\qquad
D(X^{(37,2)})&=
\begin{pmatrix}
6&0&0\\
0&6&0\\
0&0&6
\end{pmatrix},\\
A(\SI(37,2))&=
\begin{pmatrix}
1&1&1\\
1&0&2\\
1&2&0
\end{pmatrix}
,&
A(X^{(37,2)})&=
\begin{pmatrix}
2&2&2\\
2&0&4\\
2&4&0
\end{pmatrix}.
\end{alignat}
Observing the above matrices, notice that \(\SI(37,2)\) has a symmetric adjacency matrix since 0 and 1728 are not supersingular \(j\)-invariants as \(37 \equiv 1 \pmod{12}\).
\end{example}
Our concern is to investigate DSIGs running over prime numbers \(\ell\).
An important fact is that the eigenvalues of the adjacency matrix of \(\SI(r,\ell)\) are Fourier coefficients of \emph{newforms of weight 2} (meaning cuspidal normalized Hecke eigenforms in a new subspace) as we describe below.
Fix an algebraic closure \(\symbf{Q}\hookrightarrow \wideoverbar{\symbf{Q}}\) of the field \(\symbf{Q}\) of rational numbers, and choose an embedding \(\wideoverbar{\symbf{Q}}\hookrightarrow \symbf{C}\) into the field \(\symbf{C}\) of complex numbers.
In \cite{Eic_a}, Eichler shows that \(A(\SI(r,\ell))\) is a matrix representation of the Hecke operator \(T_{\ell}\) on the space \(M_2(\Gamma_0(r))\) of weight 2 modular forms of level \(r\).
\begin{proposition}[\cite{Eic_a}]
There is a basis \(\Theta_0,\Theta_1,\dotsc,\Theta_h\) of \(M_2(\Gamma_0(r))\) such that
\[
\begin{pmatrix}
T_{\ell}\Theta_0&T_{\ell}\Theta_1&\cdots &T_{\ell}\Theta_h
\end{pmatrix}
=\begin{pmatrix}
\Theta_0&\Theta_1&\cdots&\Theta_h
\end{pmatrix}
A\big(\SI(r,\ell)\big).
\]
\end{proposition}
Next, consider the decomposition
\[
M_2(\Gamma_0(r)) = \symbf{C}\cdot \symscr{E}_2 \oplus S_2(\Gamma_0(r)),
\]
where \(S_2(\Gamma_0(r))=S^{\mathrm{new}}_2(\Gamma_0(r))\) is its (new) subspace of cusp forms, and \(\symscr{E}_2\) is the unique \emph{Eisenstein series} of weight 2 and level \(r\) with \(q\)-expansion
\begin{align}\label{eq_7}
\symscr{E}_2 &= \frac{r-1}{24}+\sum_{m=1}^{\infty}\raisebox{-1.3ex}{\Bigg(}\sum_{\substack{d \smallin \symbf{N}:\\ d\mid m,\ r \not\mid d}}\hspace{-0.25em}d\raisebox{-1.3ex}{\Bigg)}\,q^m\\
&=\frac{r-1}{24}+q+\cdots+(\ell+1)q^{\ell}+\cdots.
\end{align}
Observing \eqref{eq_7}, notice that the \(\ell\)-th coefficient of \(\symscr{E}_2\) is \(\ell+1\), and hence \(T_{\ell}\symscr{E}_2 = (\ell+1) \symscr{E}_2\), which means that \(\symscr{E}_2\) gives an eigenvector for \(A(\SI(r,\ell))\) with respect to the eigenvalue \(\ell+1\).
See \cite[\S 3.2.6, \S 3.2.8, \S 3.2.9]{GL} for a brief explanation of the above fact, and \cite[\S 5]{Gro} for more details.
Considering \(A(\SI(r,\ell))\) as a matrix representation of \(T_{\ell}\), we can extract properties of \(\SI(r,\ell)\) from Eisenstein series and newforms.
\begin{proposition}[cf.\ {\cite[\S 3.2.6]{GL}}]\label{prop_1}
Let \(\gls{newforms} = \{f_1,\dotsc,f_h\}\) be the set of newforms in \(S_2(\Gamma_0(r))\).
Write \gls{fouriercoeff} for the \(\ell\)-th Fourier coefficient of a modular form \(f\).
Then \(\symscr{E}_2, f_1,\dotsc,f_h\) is a \(\symbf{C}\)-basis for \(M_2(\Gamma_0(r))\) which gives a diagonalization
\[
A(\SI(r,\ell))\sim 
\begin{pmatrix}
a_{\ell}(\symscr{E}_2)&&&\raisebox{-0.0ex}{\LARGE \(O\)}\\
&a_{\ell}(f_1)&&\\
&&\ddots&\\
\raisebox{-0.0ex}{\LARGE \(O\)}&&&a_{\ell}(f_h)
\end{pmatrix}
=
\begin{pmatrix}
\ell+1&&&\raisebox{-0.0ex}{\LARGE \(O\)}\\
&a_{\ell}(f_1)&&\\
&&\ddots&\\
\raisebox{-0.0ex}{\LARGE \(O\)}&&&a_{\ell}(f_h)
\end{pmatrix}.
\]
\end{proposition}
Proposition \ref{prop_1} is very important throughout this paper, for instance, we have the following corollary.
\begin{corollary}[cf.\ {\cite[\S 3.2.6]{GL}}]\label{cor_1}
For an arbitrary prime number \(r\) and an arbitrary prime number \(\ell\neq r\), the graph \(\SI(r,\ell)\) is connected as a directed graph.
\end{corollary}
\begin{proof}
By Lemma \ref{lem_1} below, the dimension of eigenspace of real eigenvectors of \(\ell+1\) is the number of the connected components of \(\SI(r,\ell)\).
Hence it is enough to show that no \(\ell\)-th Fourier coefficient \(a_{\ell}(f_i)\) is equal to \(\ell +1\), and this follows from Deligne's bound shown in \cite[Theorem 8.2]{Del}, which states that
\(
|a_{\ell}(f)| \leqq 2 \sqrt{\ell}
\) 
holds for every cusp form \(f\) of weight 2 and prime number \(\ell \neq r\).
\end{proof}
We describe the lemma used in the above proof.
\begin{lemma}\label{lem_1}
Let \(k\) be a positive integer and \(X\) a finite directed graph such that
\begin{enumerate}[label={\normalfont (\alph*)}]
\item the outdegree of each vertex is \(k\), and
\item if there is a path from \(u\) to \(v\), then there is a path from \(v\) to \(u\).
\end{enumerate}
Then, \(k\) is an eigenvalue of the adjacency matrix \(A(X)\), and the dimension of the eigenspace of real eigenvectors of \(k\) is equal to the number of the connected components of \(X\).
\end{lemma}
\begin{proof}
By condition (a), the vector \(\symbf{1}\) whose entries are all 1 is a \(k\)-eigenvector of \(A(X)\).
Let \(X_1,\dotsc,X_m\) be the connected components of \(X\).
Since \(A(X)_{u,v}\neq 0\) implies \(A(X)_{v,u}\neq 0\), reordering vertices of \(X\), we can write \(A(X)\) as the block diagonal matrix
\[
A(X) = \begin{pmatrix}
A(X_1) &&\text{\Large \(O\)}\\
&\ddots&\\
\text{\Large \(O\)}&&A(X_m)
\end{pmatrix}.
\]
So we have to show that, for each \(i\), the eigenspace of real eigenvectors of eigenvalue \(k\) is of dimension 1, or equivalently, the eigenspace is \(\symbf{R}\cdot\symbf{1}\).
Hence we may assume that \(X\) is a  finite connected directed graph satisfying both (a) and (b), and we will show that \(\big\{\,\symbf{x}\in \symbf{R}^{\symbf{V}(X)}:A(X)\symbf{x}=k\cdot\symbf{x}\,\big\}= \symbf{R}\cdot\symbf{1}\).
Take a vector \(\symbf{x}\in \symbf{R}^{\symbf{V}(X)}\) with \(A(X)\symbf{x}=k\cdot\symbf{x}\).
Choose a vertex \(w\in \symbf{V}(X)\) such that \(\symbf{x}_w \geqq \symbf{x}_v\) for all \(v\in \symbf{V}(X)\).
Then, for each positive integer \(n\), we have \(\sum_{v \in \symbf{V}(X)}A(X)^n_{w,v}\symbf{x}_v = k^n\symbf{x}_w\).
Moreover, since \(A(X)^n\symbf{1}=k^n\symbf{1}\), we have \(\sum_{v\in \symbf{V}(X)}A(X)^n_{w,v}=k^n\).
Therefore, we obtain an equality
\begin{equation}\label{eq_23}
\sum_{v\in\symbf{V}(X)}A(X)^n_{w,v}(\symbf{x}_w-\symbf{x}_v) = 0
\end{equation}
for each positive integer \(n\).
Note that \(A(X)^n_{w,v}\) is the number of paths of length \(n\) from \(w\) to \(v\).
Now, by connectedness of \(X\), for each \(v\), there exists a positive integer \(n\) such that \(A(X)^n_{w,v}>0\).
Thus, from \eqref{eq_23}, we deduce that \(\symbf{x}_v=\symbf{x}_w\) for each \(v\), and we have \(\symbf{x}=\symbf{x}_w \cdot\symbf{1} \in \symbf{R}\cdot \symbf{1}\), as desired.
\end{proof}
\begin{example}
Consider the case \(r=37\) again.
In example \ref{exa_1}, we have calculated \(\SI(37,2)\), and the characteristic polynomial of \(A(\SI(37,2))\) is \(T(T+2)(T-3)\).
The eigenvalue \(3=2+1\) corresponds to the Eisenstein series \(\symscr{E}_2\), and the other eigenvalues \(-2, 0\) correspond to \(\symscr{B}_{2}(37)=\{f_1,f_2\}\) respectively.
Indeed,  these newforms are given by
\begin{align}
f_1 & = q - 2q^2 -3q^3 + 2q^4 - 2q^5 + 6 q^6-q^7+6q^9+4q^{10}-5q^{11}+\cdots, \label{eq_19}\\
f_2 & = q + q^3 -2q^4-q^7-2q^9+3q^{11}+\cdots. \label{eq_8}
\end{align}
Then, the forms \eqref{eq_19}, \eqref{eq_8} tell us that \(a_2(f_1) = -2\) and \(a_2(f_2)=0\), which establish Proposition \ref{prop_1} above.
Moreover, although we have not calculated \(\SI(37,3)\), we can deduce that the eigenvalues of \(A(\SI(37,3))\) are \(4,-3,1\) by observing \eqref{eq_19} and \eqref{eq_8}.
LMFDB \cite{LMFDB} is a helpful database to find newforms in \(S_2(\Gamma_0(r))\).
\end{example}
%%%%%%%%%%%%%%%%%%%%%%%%%%%%%%%%%
\section{Big image for representations attached to newforms}\label{big}
Let \(N\) be a positive integer and \gls{absQ} denote the absolute Galois group of \(\symbf{Q}\).
In this section, we summarize the general theory for images of Galois representations attached to newforms in \(S_2^{\mathrm{new}}(\Gamma_0(N),\varepsilon)\) of level \(N\) and weight 2 with Nebentypus character \(\varepsilon\).
Our references are \cite{Loe}, \cite{Ribet_80}, \cite{Ribet_85}.
\begin{theorem}[cf.\ {\cite[Theorem 9.5.4]{DS_05}}]\label{thm_7}
For a newform \(f\in S_2^{\mathrm{new}}(\Gamma_0(N),\varepsilon)\) with Hecke field \(L\), there is a Galois representation \(\gls{rhofp}\colon  G_{\symbf{Q}}\to \GL_2\raisebox{-0.25ex}{\big(}L\otimes \symbf{Q}_p\raisebox{-0.25ex}{\big)}\) which  is unramified at each prime \(\ell \not\mid Np\), and for such \(\ell\), the linear map \(\rho_{f,p}(\mathrm{Frob}_{\ell})\) satisfies \(\tr(\rho_{f,p}(\mathrm{Frob}_{\ell})) = a_{\ell}(f)\), \(\det(\rho_{f,p}(\mathrm{Frob}_{\ell})) = \varepsilon(\ell)\ell\), where \(\mathrm{Frob}_{\ell}\) denotes a Frobenius element.
Furthermore, we can take \(\rho_{f,p}\) satisfying \(\im \rho_{f,p} \subset \GL_2(\symscr{O}_L\otimes \symbf{Z}_p)\) where \(\symscr{O}_L\) is the ring of integers of \(L\).
\end{theorem}
Henceforth, for every newform \(f\), we choose an attached Galois representation \(\rho_{f,p}\) satisfying \(\im \rho_{f,p}\subset \GL_2(\symscr{O}_L\otimes_{\symbf{Z}}\symbf{Z}_p)\).
Let \(\gls{cyclo}\colon G_{\symbf{Q}}\to \symbf{Z}_p^{\times}\) be the \(p\)-adic cyclotomic character.
Our concern is the image of the product representation 
\[
\rho_{1,p}\times\cdots \times \rho_{s,p}\times \chi_p\colon G_{\symbf{Q}}\to \GL_2(L_1\otimes \symbf{Q}_p)\times\cdots\times\GL_2(L_s\otimes\symbf{Q}_p)\times \symbf{Q}_p^{\times}
\]
of Galois representations associated to newforms \(f_1,\dotsc,f_s\) with Hecke fields \(L_1,\dotsc,L_s\) respectively.
To describe the image, we need some preliminaries.
\begin{definitions}
Let \(f\in S_2^{\mathrm{new}}(\Gamma_0(N),\varepsilon)\) be a newform, and \(L\) be the Hecke field of \(f\).
\begin{enumerate}[label=(\arabic*)]
\item For an embedding \(\sigma\colon L\hookrightarrow \wideoverbar{\symbf{Q} }\), the \emph{Galois conjugate} \(f^{\sigma}\) is the newform in \(S_2(\Gamma_0(N),\varepsilon^{\sigma})\) whose Fourier coefficients are given by \(a_{n}(f^{\sigma}) = \sigma(a_{n}(f))\) for all positive integers \(n\).
A set of the form \(\big\{\,f^{\sigma}:\text{\(\sigma\) is an embedding of \(L\) into \(\wideoverbar{\symbf{Q}}\)}\,\big\}\) is called a \emph{newform orbit}.
\item For a Dirichlet character \(\chi\colon \symbf{Z} \to \wideoverbar{\symbf{Q}}^{\times}\), the \emph{twist} \(f\otimes \chi\) is the newform uniquely determined by the formula \(a_{\ell}(f\otimes \chi) = \chi(\ell) a_{\ell}(f)\) for all but finitely many prime \(\ell\).
\item \(f\) \emph{has complex multiplication} (in the sense of \cite{Ribet_80}) if \(f=f\otimes \chi\) for some non-trivial Dirichlet character \(\chi\).
By \cite[(3.10)]{Ribet_80}, if \(N\) is square-free and the Nebentypus character is trivial, then \(f\) has no complex multiplication.
\item An \emph{inner twist} of \(f\) is a pair \((\gamma,\chi)\) of an embedding \(\gamma\colon L\hookrightarrow \wideoverbar{\symbf{Q}}\) and a Dirichlet character \(\chi\) such that \(f^{\gamma}=f\otimes \chi\).
For an inner twist \((\gamma,\chi)\), Momose shows in \cite[Lemma 1.5 (i)]{Momose} that \(\gamma\) induces an element of \(\Aut(L/\symbf{Q})\), and this element is called an \emph{extra twist}.
If \(f\) has no complex multiplication, then we find that \(\chi\) is uniquely determined by \(\gamma\), that is, for two inner twists \((\gamma,\chi), (\gamma,\chi')\), we must have \(\chi=\chi'\).
\end{enumerate}
\end{definitions}
Now we are ready to state the following big image theorem for a finite collection of newforms, which is finite product version of \cite[Theorem 3.2.2]{Loe}.
\begin{theorem}[Lei, Loeffler, and Zerbes {\cite[\S 7.2.2]{LLZ}}, Loeffler {\cite[Theorem 3.2.2]{Loe}}] \label{thm_4}
Let \(f_1,\dotsc,f_s\) be newforms of level \(N\) and weight 2 that have no complex multiplication.
Write \(H_k\) for the open subgroup of \(G_{\symbf{Q}}\) which is the intersection of the kernels of the Dirichlet characters \(\chi_{\gamma}\) determined by extra twists \(\gamma\) of \(f_k\), and \(F_k\) for the number field fixed by the extra twists of \(f_k\).
Let \(\symscr{O}_{F_k}\) denote the ring of integers of \(F_k\).
Suppose that there do not exist a Dirichlet character \(\chi\) and indices \(k,l \in \{1,\dotsc,s\}\) such that \(f_k\otimes \chi\) is Galois conjugate to \(f_l\) with \(\chi\) non-trivial if \(k=l\).
Then, for all but finitely many \(p\), we have 
\begin{equation}\label{eq_10}
\raisebox{-0.5ex}{\big(}\rho_{1,p}\times\cdots\times\rho_{s,p}\times\chi_p\raisebox{-0.5ex}{\big)}\big(H_1\cap \cdots\cap H_s\big)= \big\{\,(x_1,\dotsc,x_s,t):\text{\(x_k\in \GL_2(\symscr{O}_{F_k}\otimes \symbf{Z}_p)\), \(\det(x_k)=t\) for all \(k\)}\,\big\}.
\end{equation}
\end{theorem}
By \cite[(3.9 bis)]{Ribet_80}, if \(\varepsilon_f\) is trivial and \(N\) is square-free, then the level of \(f\otimes \chi\) is not square-free, and by \cite[(3.10)]{Ribet_80}, \(f\) with the same assumption has no complex multiplication.
As a consequence, a newform \(f\in S_2(\Gamma_0(r))\), which is our concern, has no non-trivial inner twist and no complex multiplication.
Hence we have the following corollary.
\begin{corollary}\label{cor_2}
Let \(\symscr{B}_2(r)=\{f_1,\dotsc,f_h\}\) be the set of newforms in \(S_2(\Gamma_0(r))\), and \(\symscr{O}_i\) be the ring of integers of the Hecke field \(L_i\).
Renumbering if necessary, let \(f_1,\dotsc,f_s\) be a complete set of representatives of the newform orbits.
Then we have for all but finitely many prime numbers \(p\),
\begin{equation}\label{eq_11}
\im\raisebox{-0.5ex}{\big(}\rho_{1,p}\times\cdots\times\rho_{s,p}\times \chi_p \raisebox{-0.5ex}{\big)}= \big\{\,(x_1,\dotsc,x_s,t):\text{\(x_k\in \GL_2(\symscr{O}_{k}\otimes \symbf{Z}_p)\), \(\det(x_k)=t\) for each \(k\)}\,\big\}.
\end{equation}
\end{corollary}
\begin{proof}
Since the Nebentypus character of \(f_i\) is trivial and \(r\) is square-free, \(f_i\) has no complex multiplication, and the twist \(f_i\otimes \chi\) by a non-trivial Dirichlet character is of level different from \(r\).
Hence we can apply Theorem \ref{thm_4} to \(f_1,\dotsc,f_s\).
Since each \(f_i\) has no non-trivial inner twist, \(H_i\) is the whole group \(G_{\symbf{Q}}\) and \(F_i\) is the whole field \(L_i\).
Therefore \eqref{eq_10} is nothing but \eqref{eq_11}.
\end{proof}
%%%%%%%%%%%%%%%%%%%%%%%%%%%%%%%%%%%%%%%%%%%%%%%%%
\section{Results}\label{sec_res}
\subsection{Iwasawa invariants of supersingular isogeny graphs}\label{sec_a}
From now on, assume that \(\gls{r}\equiv 1 \pmod{12}\) so that 0 and 1728 are not supersingular \(j\)-invariants.
Set \(h+1\coloneq \lfloor r/12\rfloor\), give a numbering of a basis of newforms as \(\symscr{B}_2(r)=\{f_1,\dotsc,f_h\}\), and put \(f_0\coloneq \symscr{E}_2\), the Eisenstein series of weight 2 and level \(r\).
For each \(i=1,\dotsc,h\), put \(L_i\coloneq \symbf{Q}(\{\,a_{n}(f_i):n\in \symbf{N}\,\})\), which is known as the \emph{coefficient field} or the \emph{Hecke field of \(f_i\)} and is a number field (see \cite[\S 6.5]{DS_05} for more detail).
Consider the composite field \(L\coloneq L_1\cdots L_h\).
Note that, since newforms are closed under Galois conjugation, \(L\) is Galois over \(\symbf{Q}\).
Let \(\symscr{O}_L\) denote the ring of integers of \(L\).
For each odd prime number \gls{p}, choose a prime ideal \(\symfrak{P}\) of \(\symscr{O}_L\) lying above \(p\).
Let \(L_{\symfrak{P}}\) be a completion of \(L\) by \(\symfrak{P}\), and let \(\symscr{O}_{\symfrak{P}}\) be its valuation ring.
Fix an algebraic closure \(L_{\symfrak{P}}\hookrightarrow \wideoverbar{\symbf{Q}}_p\) of \(L_{\symfrak{P}}\).
Set \(\Lambda\coloneq \symscr{O}_{\symfrak{P}}\lBrack T\rBrack\) and embed \(\Lambda \hookrightarrow \Lambda\otimeshat_{\symscr{O}_{\symfrak{P}}}\wideoverbar{\symbf{Q}}_{p}\cong \wideoverbar{\symbf{Q}}_{p}\lBrack T\rBrack\).

Now fix an odd prime number \(p\) and consider the \(\symbf{Z}_p\)-tower
\begin{equation}
X^{(r,\ell)} \leftarrow X^{(r,\ell)}_1 \leftarrow \cdots \leftarrow X^{(r,\ell)}_n \leftarrow \cdots
\end{equation}
over \(X^{(r,\ell)}\) induced by the voltage assignment \(\alpha\colon \symbf{E}(X^{(r,\ell)})\to \symbf{Z}_p\);
\begin{equation}
\symbf{E}(\SI(r,\ell))\ni e \mapsto 1 \in \symbf{Z}_p, \qquad \wideoverbar{\symbf{E}(\SI(r,\ell))}\ni e \mapsto -1 \in \symbf{Z}_p,
\end{equation}
so the \(n\)-th layer \(X^{(r,\ell)}_n\) is the derived graph of the voltage assignment
\[
\alpha_n\colon \symbf{E}(X^{(r,\ell)})\to \symbf{Z}_p/p^n\symbf{Z}_p;\ e\mapsto \alpha(e)+p^n\symbf{Z}_p.
\]
We call the above  \(\symbf{Z}_p\)-tower the \emph{constant \(\symbf{Z}_p\)-tower over \(X^{(r,\ell)}\)}, such a tower is considered in \cite{LM_c}; see Appendix A for a precise explanation.
Recall that being a \(\symbf{Z}_p\)-tower requires connectedness of \(X^{(r,\ell)}_n\) for each \(n\).
To guarantee this condition, we utilize Definition-Proposition \ref{def-prop_1} (2).
If \(h=0\), that is, if \(r = 13\), then \(X^{(r,\ell)}\) consists of a single vertex with \(\ell+1\)-self-isogenies.
1 is assigned to each self-isogeny, and it generates \(\symbf{Z}_p/p^n\symbf{Z}_p \cong \symbf{Z}/p^n\symbf{Z}\).
If \(h>0\), then an isogeny between distinct elliptic curves and its dual isogeny form a cycle of length 2, so 2 is assigned to this cycle, and it generates \(\symbf{Z}_p/p^n\symbf{Z}_p\) since \(p\) is odd.
Thus, in both cases, \(X^{(r,\ell)}_n\) is connected for each \(n\).
\begin{remark}[cf.\ {\cite[Corollary 3.7]{LM_c}}]
Alternatively, we can consider another constant \(\symbf{Z}_p\)-tower by giving a voltage assignment \(\alpha'\) as
\[
\symbf{E}(\SI(r,\ell))\ni e \mapsto c \in \symbf{Z}_p, \qquad 
\wideoverbar{\symbf{E}(\SI(r,\ell))}\ni e \mapsto -c \in \symbf{Z}_p
\]
for a unit \(c\in \symbf{Z}_p^{\times}\), but for each number \(n\), the maps
\[
\symbf{V}\big(X(\alpha_n)\big)\ni (v,\sigma)\mapsto (v,c\sigma)\in \symbf{V}\big(X(\alpha'_n)\big), \qquad \symbf{E}\big(X(\alpha_n)\big)\ni (e,\sigma)\mapsto (e,c\sigma)\in \symbf{E}\big(X(\alpha'_n)\big)
\]
induce an isomorphism from \(X(\alpha_n)\) to \(X(\alpha'_n)\) which is compatible with respect to \(\symbf{Z}_p\)-towers.
In particular, Iwasawa invariants of \(X(\alpha_n)\) and \(X(\alpha'_n)\) are the same.
Hence we only deal with our constant \(\symbf{Z}_p\)-towers \(X(\alpha_n)\).
\end{remark}
\begin{definition}
Consider the characteristic element of the \(\symbf{Z}_p\)-tower over \(X^{(r,\ell)}\)
\begin{alignat}{3}
\gls{characteristicelement}
&\overset{\eqref{eq_6}}{\coloneq\joinrel=\joinrel=\joinrel=}&\,& \det \raisebox{-0.75ex}{\bigg(}D(X^{(r,\ell)})-\raisebox{-0.75ex}{\bigg(}\sum_{e\smallin \symbf{E}(X^{(r,\ell)})_{E,E'}} (1+T)^{\alpha(e)}\raisebox{-0.75ex}{\(\bigg)_{E,E'\smallin \symbf{V}(X^{(r,\ell)})}\)}\raisebox{-0.75ex}{\bigg)}\\
&=\joinrel=\joinrel=\joinrel=&&\det (D(X^{(r,\ell)})-(1+T)A(\SI(r,\ell))-(1+T)^{-1}A\raisebox{0.125ex}{\big(}\wideoverbar{\SI}(r,\ell))\raisebox{0.125ex}{\big)}\\
&\overset{\text{Prop.\ \ref{prop_2}}}{=\joinrel=\joinrel=\joinrel=}&&\det (D(X^{(r,\ell)})-(1+T)A(\SI(r,\ell))-(1+T)^{-1}A(\SI(r,\ell)))\\
&=\joinrel=\joinrel=\joinrel=&&\det \Big(2D\big(\SI(r,\ell)\big)-\big(2+T^2(1+T)^{-1}\big)A(\SI(r,\ell))\Big) \in \Lambda.
\end{alignat}
For each \(i=0,1,\dotsc,h\), put
\begin{equation}
\label{eq_2}
\gls{characteristicfactor}\coloneq \ell+1-a_{\ell}(f_i)-\frac{a_{\ell}(f_i)}{2}T^2(1+T)^{-1} \in \Lambda.
\end{equation}
\end{definition}
The main theorem in this section is the following.
\begin{theorem}\label{thm_3}
Let \(r\) be a prime number with \(r\equiv 1\pmod{12}\), \(\ell\) a prime number with \(\ell \neq r\), and  \(p\) an odd prime number.
Then, for the characteristic element \(\Delta_{\ell}\) of the constant \(\symbf{Z}_p\)-tower over \(X^{(r,\ell)}\), the following hold.
\begin{enumerate}[label=\((\arabic*)\)]
\item \(\Delta_{\ell}(T)=2^{h+1}\prod_{i=0}^h\Delta_{\ell,i}\).
\end{enumerate}
In the following two statements, we further assume that \(\ell + 1 \not\equiv 0 \pmod{p}\).
This condition is satisfied, for instance, if \(p\) is odd and \(\ell\equiv 1\pmod{p}\).
\begin{enumerate}[label=\((\arabic*)\)]
\setcounter{enumi}{1}
\item \(\mu(\Delta_{\ell}) = 0\),
\item \(\lambda(\Delta_{\ell})=2|\{\,i : \ell + 1 - a_{\ell}(f_i) \equiv 0 \pmod{\symfrak{P}}\,\}|\).
In particular, since \(\ell+1-a_{\ell}(f_0)=0\), we deduce that \(\lambda(\Delta_{\ell})>0\).
\end{enumerate}
\end{theorem}
\begin{proof}
(1) For simplicity, we put \(D_{\ell}\coloneq D(\SI(r,\ell))\), \(A_{\ell}\coloneq A(\SI(r,\ell))\) in this calculation.
We know that \(A_{\ell}\) is a real symmetric matrix, so the diagonalization in Proposition \ref{prop_1} can be done by an orthogonal matrix \(P_{\ell}\), that is, 
\begin{equation}\label{eq_1}
P_{\ell}^{\symsf{T}}A_{\ell}P_{\ell} = 
\begin{pmatrix}
\ell+1&&&\raisebox{-0.0ex}{\LARGE \(O\)}\\
&a_{\ell}(f_1)&&\\
&&\ddots&\\
\raisebox{-0.0ex}{\LARGE \(O\)}&&&a_{\ell}(f_h)
\end{pmatrix}.
\end{equation}
Then \(\Delta_{\ell}\) can be calculated as, in \(\Lambda\otimes_{\symscr{O}_{\symfrak{P}}}\wideoverbar{\symbf{Q}}_p\), 
\begin{align}
\Delta_{\ell}(T)&=\det \Big(2D_{\ell}-\big(2+T^2(1+T)^{-1}\big)A_{\ell}\Big)\\
&=\det (P_{\ell}^{\symsf{T}})\det\Big(2D_{\ell}-\big(2+T^2(1+T)^{-1}\big)A_{\ell}\Big)\det(P_{\ell})
\\
&=\det \Big(2D_{\ell}-\big(2+T^2(1+T)^{-1}\big)P^{\symsf{T}}A_{\ell}P\Big)\\
&\overset{\textcircled{\scriptsize\(\clubsuit\)}}{=}\prod_{i=0}^h \Big(2(\ell+1)-a_{\ell}(f_i)\big(2+T^2(1+T)^{-1}\big)\Big)
\\
&=2^{h+1}\prod_{i=0}^h\Big(\ell + 1 - a_{\ell}(f_i)-\frac{a_{\ell}(f_i)}{2}T^2(1+T)^{-1}\Big)\\
&=2^{h+1}\prod_{i=0}^h \Delta_{\ell,i}(T).
\end{align}
For the equality under \textcircled{\scriptsize\(\clubsuit\)}, we use the identity \(D_{\ell}=(\ell+1)I_{h+1}\) where \(I_{h+1}\) denotes the identity matrix of size \(h+1\).

\indent (2) Proof by contradiction.
Suppose \(\mu(\Delta_{\ell,i})> 0\) for some \(i\).
This means that \(\Delta_{\ell,i}(T) \in \symfrak{P}\Lambda\), both \(a_{\ell}(f_i)\) and \(\ell +1 -a_{\ell}(f_i)\) belong to \(\symfrak{P}\) in particular, and therefore \(\ell +1\equiv 0 \pmod{\symfrak{P}}\), which contradicts our choice of \(\ell\).
Hence \(\mu(\Delta_{\ell,i})= 0\) for all \(i\), and so \(\mu(\Delta_{\ell})= 0\).

\indent (3)
Let \(\wideoverbar{\hspace{0pt}\Delta}_{\ell}\) denote the image of \(\Delta_{\ell}\) in the quotient \(\Lambda/\symfrak{P}\Lambda\).
By definition of \(\lambda\)-invariants, \(\lambda(\Delta_{\ell})\) is the degree of the lowest monomial in \(\wideoverbar{\hspace{0pt}\Delta}_{\ell}\), and hence \(\lambda(\Delta_{\ell})=\sum_{i=0}^h \lambda(\Delta_{\ell,i})\).
Observing the expression \eqref{eq_2}, we have
\begin{equation}\label{eq_4}
\lambda(\Delta_{\ell,i})=\begin{cases}
0& \text{if \(\ell + 1 -a_{\ell}(f_i) \not\equiv 0 \pmod{\symfrak{P}}\)},\\
2& \text{if \(\ell + 1 -a_{\ell}(f_i) \equiv 0 \pmod{\symfrak{P}}\)}.
\end{cases}
\end{equation}
In the latter case, we notice that if \(\ell+1-a_{\ell}(f_i)\equiv 0 \pmod{\symfrak{P}}\), then we must have \(a_{\ell}(f_i) \not\equiv 0\pmod{\symfrak{P}}\) because \(\ell + 1\not\equiv 0 \pmod{\symfrak{P}}\) as in the proof of (2).
Combining this calculation with (1), we obtain the desired identity.
\end{proof}
%%%%%%%%%%%%%%%%%%%%%%%%%%%%%%%%%%%%%%
\subsection{Distribution theorem}
As in the section \ref{sec_a}, for each \(i\in\{1,\dotsc,h\}\), \(L_i\) denotes the Hecke field of the newform \(f_i\), i.e., the field generated  by the Fourier coefficients of \(f_i\), \(\symscr{O}_i\) denotes the ring of integers of \(L_i\), \(L\) denotes the composite field \(L_1\cdots L_h\), and \(\symscr{O}_L\) denotes the ring of integers of \(L\).
Now renumber \(\symscr{B}_2(r)\) as in corollary \ref{cor_2}, so \(\{f_1,\dotsc,f_s\}\) is a set of complete representatives of the newform orbits.
For each \(k\in \{1,\dotsc,s\}\), write \(O(k)\) for the newform orbit of \(f_k\), so \([L_k:\symbf{Q}] = |O(k)|\) and \(\symscr{B}_2(r)=O(1)\sqcup\cdots \sqcup O(s)\).
Let \(L_{O(k)}\) denote the composite field of the Hecke fields of the newforms in \(O(k)\), so \(L_{O(k)}/\symbf{Q}\) is the Galois closure of \(L_k/\symbf{Q}\).
We also write \(\symscr{O}_{O(k)}\) for the ring of integers of \(L_{O(k)}\).

Choose an odd prime \(p\) so that the identity \eqref{eq_11} holds, and split \(p\) as 
\begin{align}
p\symscr{O}_L & = \big(\symfrak{P}_1\cdots \symfrak{P}_{g_L}\big)^{e_L},\\
p\symscr{O}_{O(k)} &= \raisebox{-0.25ex}{(}\symfrak{p}_{k,1}\cdots \symfrak{p}_{k,g_{O(k)}}\raisebox{-0.25ex}{)}^{e_k},\\
p\symscr{O}_k &= \symcal{p}_{k,1}^{c_{k,1}}\cdots \symcal{p}_{k,g_k}^{c_{k,g_k}}.
\end{align}
Renumbering if necessary, we assume that \(\symfrak{P} = \symfrak{P}_1\), \(\symfrak{P}\cap \symscr{O}_{O(k)}=\symfrak{p}_{k,1}\), and \(\symfrak{P}\cap \symscr{O}_k=\symcal{p}_{k,1}\), where \(\symfrak{P}\) is the prime ideal of \(\symscr{O}_L\) chosen in the previous subsection.
Under this numbering, we take canonical isomorphisms
\[
L_k\otimes_{\symbf{Q}}\symbf{Q}_p \cong \prod_{i=1}^{g_k}L_{k,\symcal{p}_{k,i}},\qquad 
\symscr{O}_k\otimes_{\symbf{Z}}\symbf{Z}_p \cong \prod_{i=1}^{g_k}\symscr{O}_{k,\symcal{p}_{k,i}},
\]
where \(\symscr{O}_{k,\symcal{p}_{k,i}}\) denotes the ring of integers of \(L_{k,\symcal{p}_{k,i}}\).
Let \(\bar{\rho}_p\) denote the representation obtained by composing \(\rho_p\) with the product of natural projections
\[
\GL_2(\symscr{O}_k\otimes_{\symbf{Z}}\symbf{Z}_p) \cong \GL_2\Bigg(\prod_{i=1}^{g_k}\symscr{O}_{k,\symcal{p}_{k,i}}\Bigg)\twoheadrightarrow \GL_2\Bigg(\prod_{i=1}^{g_k}\symbf{F}_{k,\symcal{p}_{k,i}}\Bigg)\qquad (k=1,\dotsc,s)
\]
and \(\symbf{Z}_p^{\times}\twoheadrightarrow \symbf{F}_p^{\times}\), where \(\symbf{F}_{\symcal{p}_{k,i}}\) denotes the residue field of \(\symcal{p}_{k,i}\).
We remark that one can also define \(\bar{\rho}_p\) as the product representation
\[
\bar{\rho}_{1,p}\times\cdots \times \bar{\rho}_{s,p}\times \bar{\chi}_p\colon G_{\symbf{Q}}\to \GL_2\raisebox{0.25ex}{\Bigg(}\prod_{i=1}^{g_1}\symbf{F}_{\symcal{p}_{1,i}}\raisebox{0.25ex}{\Bigg)}\times\cdots\times\GL_2\raisebox{0.25ex}{\Bigg(}\prod_{i=1}^{g_s}\symbf{F}_{\symcal{p}_{s,i}}\raisebox{0.25ex}{\Bigg)}\times \symbf{F}_p^{\times}.
\]
Let \(K\) denote the intermediate field corresponding to the normal closed subgroup \(\ker \bar{\rho}_{p} \vartriangleleft  G_{\symbf{Q}}\).
Then, by the homomorphism theorem, we have \(\im \bar{\rho}_p\cong G_{\symbf{Q}}/\ker \bar{\rho}_{p}\cong\Gal(K/\symbf{Q})\), and hence \(K/\symbf{Q}\) is a finite extension.
We know that \(\bar{\rho}_{p}\) is unramified at \(\ell\neq rp\), and such \(\ell\) is unramified in \(K/\symbf{Q}\).

Before stating our main theorem, we recall Chebotarev density theorem here.
\begin{theorem}[Chebotarev Density Theorem]\label{thm_1}
Let \(K/\symbf{Q}\) be a Galois extension of number fields with Galois group \(G\).
Consider a non-empty subset \(C \subset G\) stable under conjugation, where \textup{``\emph{stable under conjugation}''} means that \(g^{-1}Cg\subset C\) for all \(g\in G\).
Then, the set 
\begin{equation}\label{eq_3}
\gls{frakLC}\coloneq \{\,
\ell:\text{\(\ell\) is a prime number, \(\ell\) is unramified in \(K/\symbf{Q}\), \(\mathrm{Frob}_{\ell}\in C\)}
\,\}
\end{equation}
has density \(|C|/|G|\).
\end{theorem}
\begin{maintheorem}[Distribution theorem]\label{thm_2}
Let \(r\) be a prime number with \(r\equiv 1 \pmod{12}\).
For \(\ell\), write \(\lambda_{\ell}\) for the Iwasawa \(\lambda\)-invariant of the constant \(\symbf{Z}_p\)-tower over \(X^{(r,\ell)}\).
Then, for all but finitely many \(p\), the following statement holds.
For an arbitrary subset \(W \subset \{1,\dotsc,s\}\),
the set
\[
\symscr{L}_W\coloneq \bigg\{\,\ell:
\text{
\(\ell\) is a prime number distinct from \(r\), \(\ell\equiv 1 \pmod{p}\), \(\lambda_{\ell}=1+2\sum_{k\smallin W}|O(k)|\)
}
\,\bigg\}
\]
has positive lower density.
\end{maintheorem}
\begin{proof}
We fix an odd prime number \(p\) such that the statement of Corollary \ref{cor_2} holds.
We put \(G\coloneq \Gal(K/\symbf{Q})\) and identify \(\Gal(K/\symbf{Q})\) with \(\im \bar{\rho}_p\).
Consider the subset
\[
C_W \coloneq \left\{\,\sigma \in G:\begin{minipage}{23em}
\(\bar{\chi}_p(\sigma)= 1\), \(\det\big(I_2-\bar{\rho}_{k,p}(\sigma)\big)= (0,\dotsc,0)\) for \(k \in W\), \(\det\big(I_2-\bar{\rho}_{k,p}(\sigma)\big)\in \prod_{i=1}^{g_k}\symbf{F}_{\symcal{p}_{k,i}}^{\times}\) for \(k \not\in W\)
\end{minipage}\,\right\},
\]
where \(I_2\) stands for the identity matrix of size \(2\).
We will first show that \(C_W\) is non-empty and stable under conjugation, and then show that \(\symfrak{L}(C_W)\subset \symscr{L}_W\) and apply Theorem \ref{thm_1} (Chebotarev Density Theorem).

Reducing \eqref{eq_11} modulo \(p\), there exists an element \(\sigma_0\in G\) such that \(\bar{\rho}_{k,p}(\sigma_0)\) is \(I_2\) for \(k\in W\), \(\bar{\rho}_{k,p}(\sigma_0)\) is \(-I_2\) for \(k\not\in W\), and \(\bar{\chi}_p(\sigma_0)=1\).
For such \(\sigma_0\), we have \(\sigma_0\in C_W\) and hence \(C_W\neq\emptyset\).
Next, we will confirm the stability of \(C_W\) under conjugation.
For each \(\sigma_1\in C_W\) and every \(\sigma_2 \in G\), we have
\begin{equation}
\bar{\chi}_p(\sigma_2^{-1}\sigma_1\sigma_2) =\bar{\chi}_p(\sigma_2)^{-1}\bar{\chi}_p(\sigma_1)\bar{\chi}_p(\sigma_2)
=\bar{\chi}_p(\sigma_1)
=1,
\end{equation}
for the former condition, and
\begin{align}
\det \big(I_2-\bar{\rho}_{k,p}&(\sigma_2^{-1}\sigma_1\sigma_2)\big)\\
&=\det \big(\bar{\rho}_{k,p}(\sigma_2)\big)^{-1}\det \big(I_2-\bar{\rho}_{k,p}(\sigma_1)\big)\det \big(\bar{\rho}_{k,p}(\sigma_2)\big)\\
&=\det \big(I_2-\bar{\rho}_{k,p}(\sigma_1)\big)\\
&\begin{cases}=(0,\dotsc,0) &\text{for \(k\in W\),}\\
\in \prod_{i=1}^{g_k}\symbf{F}_{\symcal{p}_{k,i}}^{\times} &\text{for \(k\not\in W\),}\end{cases}\label{eq_12}
\end{align}
for the latter condition, and thus we have proven that \(C_W\) is a non-empty subset stable under conjugation.
Hence we are able to apply Chebotarev Density Theorem (Theorem \ref{thm_1}) to \(C_W\), and so \(\symfrak{L}(C_W)\) has positive density \(|C_W|/|G|\).

Noticing that \(r\) may belong to \(\symfrak{L}(C_W)\), it remains to establish the inclusion \(\symfrak{L}(C_W)\setminus\{r\} \subset \symscr{L}_W\).
To begin with, \(\bar{\chi}_p(\mathrm{Frob}_{\ell}) = 1\) means that  \(\mathrm{Frob}_{\ell}(\zeta_p)=\zeta_p\), that is, the element
\(
\zeta_p^{\ell}-\zeta_p=\zeta_p(\zeta_p^{\ell-1}-1)
\)
is zero.
Therefore, we deduce that 
\(\bar{\chi}_p(\mathrm{Frob}_{\ell})=1\) if and only if  \(\ell\equiv 1 \pmod{p}\).
The hard part is to control the \(\lambda\)-invariants.
Observe that, for each \(k\in \{1,\dotsc,s\}\) and each prime \(\ell\not\mid rp\), we have
\begin{align}
\det\Big(I_2-\bar{\rho}_{k,p}(\mathrm{Frob}_{\ell})\Big)& \overset{\textcircled{\scriptsize\(\spadesuit\)}}{=}1-\tr\Big(\bar{\rho}_{k,p}(\mathrm{Frob}_{\ell})\Big)+\det\Big(\bar{\rho}_{k,p}(\mathrm{Frob}_{\ell})\Big)
\\
&\overset{\textcircled{\raisebox{-0.125ex}{\scriptsize\(\varheartsuit\)}}}{=}\Big(\ell + 1 - a_{\ell}(f_k) \pmod{\symcal{p}_{k,i}}\Big)_{i=1}^{g_k}. \label{eq_13}
\end{align}
The equality under \textcircled{\scriptsize\(\spadesuit\)} holds for general square matrices of size 2, and the equality under \textcircled{\raisebox{-0.125ex}{\scriptsize\(\varheartsuit\)}} holds since \(\det\big(\rho_{k,p}(\mathrm{Frob}_{\ell})\big)=\ell\) and \(\tr\big(\rho_{k,p}(\mathrm{Frob}_{\ell})\big)=a_{\ell}(f_k)\) by Theorem \ref{thm_7}.
Consider a prime number \(\ell\) such that \(\ell\not\mid rp\) and \(\mathrm{Frob}_{\ell}\in C_W\).
Then, by the definition of \(C_W\) and \eqref{eq_13},
\begin{alignat}{2}
&\ell+1-a_{\ell}(f_k) \in \symcal{p}_{k,1}\cap \cdots \cap \symcal{p}_{k,g_k}&\qquad&\text{if \(k\in W\)},\label{eq_14}\\
&\ell+1-a_{\ell}(f_k) \not\in \symcal{p}_{k,1}\cup \cdots \cup \symcal{p}_{k,g_k}&&\text{if \(k\not\in W\)},\label{eq_15}
\end{alignat}
in particular,
\begin{alignat}{2}
&\ell+1-a_{\ell}(f_k) \in \symfrak{P}_1&\qquad&\text{if \(k\in W\)},\label{eq_20}\\
&\ell+1-a_{\ell}(f_k) \not\in \symfrak{P}_1&&\text{if \(k\not\in W\)}.\label{eq_21}
\end{alignat}
Our desire is to investigate \(\ell+1-a_{\ell}(f_i)\) for each remaining number \(i \in \{s+1,\dotsc,h\}\).
The newform \(f_i\) is the Galois conjugate \(f_k^{\sigma}\) for some unique \(k\in\{1,\dotsc,s\}\) and some \(\sigma \in \Gal(L_{O(k)}/\symbf{Q})\).
For such \(k\), by \eqref{eq_14} and \eqref{eq_15}, we have the prime decomposition
\begin{equation}
\big(\ell + 1 - a_{\ell}(f_k)\big)\symscr{O}_{O(k)}=\symfrak{p}_{k,1}^{d_1}\cdots\symfrak{p}_{k,g_{O(k)}}^{d_{g_{O(k)}}}\cdot\symfrak{a},
\end{equation}
with \(d_1,\dotsc,d_{g_{O(k)}}\geqq 1\) if \(k\in W\), and with \(d_1=\cdots = d_{g_{O(k)}}=0\) if \(k\not\in W\).
Since \(L_{O(k)}/\symbf{Q}\) is Galois and \(a_{\ell}(f_i) = \sigma\big(a_{\ell}(f_k)\big)\), we thus have
\begin{equation}
\big(\ell + 1 - a_{\ell}(f_i)\big)\symscr{O}_{O(k)}=\symfrak{p}_{k,1}^{d'_1}\cdots\symfrak{p}_{k,g_{O(k)}}^{d'_{g_{O(k)}}}\cdot\sigma(\symfrak{a}),
\end{equation}
where \(d'_1,\dotsc,d'_{g_{O(k)}}\) is the permutation of \(d_1,\dotsc,d_{g_{O(k)}}\) by the action of \(\sigma\) on \(\symfrak{p}_{k,1},\dotsc,\symfrak{p}_{k,g_{O(k)}}\).
This means that 
\begin{alignat}{3}
&1+\ell-a_{\ell}(f_i) \in \symfrak{P}_1 &\qquad& \text{if \(k \in W\)},\label{eq_16}\\
&1+\ell-a_{\ell}(f_i) \not\in \symfrak{P}_1&& \text{if \(k \not\in W\)}\label{eq_17}.
\end{alignat}
Consequently, we can calculate
\begin{alignat}{3}
\lambda_{\ell}&\overset{\text{Prop.\ \ref{prop_3}}}{=\joinrel=\joinrel=\joinrel=\joinrel=\joinrel=\joinrel=\joinrel=\joinrel=}&& \lambda(\Delta_{\ell})-1\\
&\overset{\text{Thm.\ \ref{thm_3} (3)}}{=\joinrel=\joinrel=\joinrel=\joinrel=\joinrel=\joinrel=\joinrel=\joinrel=}&& 2|\{\,i\in\{0,1,\dotsc,h\}: \ell + 1 - a_{\ell}(f_i) \equiv 0 \pmod{\symfrak{P}_1}\,\}|-1\\
&\overset{\text{\eqref{eq_20}, \eqref{eq_21}, \eqref{eq_16}, \eqref{eq_17}}}{=\joinrel=\joinrel=\joinrel=\joinrel=\joinrel=\joinrel=\joinrel=\joinrel=}&\quad& 2\raisebox{-0.5ex}{\bigg(}1+\sum_{k\smallin W}|O(k)|\raisebox{-0.5ex}{\bigg)}-1\\
&=\joinrel=\joinrel=\joinrel=\joinrel=\joinrel=\joinrel=\joinrel=\joinrel=&& 1+2\sum_{k\smallin W}|O(k)|.
\end{alignat}
As a consequence, if \(\mathrm{Frob}_{\ell}\in C_W\), then \(\ell\equiv 1 \pmod{p}\) and \(\lambda_{\ell}=1+2\sum_{k\smallin W}|O(k)|\).
Recalling that \(r\) may belong to \(\symfrak{L}(C_W)\), we have
\[
\symfrak{L}(C_W)\setminus \{r\}\subset
\Bigg\{\,\ell:
\begin{minipage}{22em}
\(\ell\) is a prime number distinct from \(r\), \(\ell\equiv 1 \pmod{p}\), \(\lambda_{\ell}=1+2\sum_{k\smallin W}|O(k)|\)
\end{minipage}
\,\Bigg\}
\subset \symscr{L}_W,
\]
and thus \(\symfrak{L}(C_W) \setminus \{r\}\subset \symscr{L}_W\).
Since \(\symfrak{L}(C_W)\) and \(\symfrak{L}(C_W)\setminus \{r\}\) have the same lower density, we have
\[
0 < |C_W|/|G| \overset{\text{Thm.\ \ref{thm_1}}}{=\joinrel = \joinrel = \joinrel =} \big(\text{lower density of \(\symfrak{L}(C_W)\)}\big) \leqq \big(\text{lower density of \(\symscr{L}_W\)}\big),
\]
as desired.
\end{proof}
As we mentioned above, since \([L_k:\symbf{Q}]\) is the number of embeddings of \(L_k\) into \(\wideoverbar{\symbf{Q}}\), one has \(|O(k)|=[L_k:\symbf{Q}]\), and utilizing a database of newforms LMFDB \cite{LMFDB}, we can find the extension degree of Hecke fields over \(\symbf{Q}\) and hence the size of the newform orbits.
Also note that odd numbers of the form \(1+2\sum_{k\smallin W}|O(k)|\) depend only on \(r\).
Hence, we phrase Main Theorem \ref{thm_2} as follows.
Give a prime number \(r\) with \(r\equiv 1  \pmod{12}\), and consider the newform orbits \(O(1),\dotsc,O(s)\).
Then, for an arbitrary subset \(W \subset \{1,\dotsc,s\}\) and the odd number 
\begin{equation}\label{eq_18}
1+2\sum_{k \smallin W}|O(k)|,
\end{equation}
the prime numbers \(\ell\) such that the Iwasawa \(\lambda\)-invariant of the constant \(\symbf{Z}_p\)-tower over \(X^{(r,\ell)}\) is equal to \(1+2\sum_{k\smallin W}|O(k)|\) have positive lower density for all but finitely many odd prime numbers \(p\).
In Table \ref{tab_1}, we list the prime numbers \(r \equiv 1 \pmod{12}\) less than 2000 and odd numbers of the form \eqref{eq_18}.
Note that there are all odd numbers not greater than 200 except for 83 and 123.
Furthermore, according to LMFDB \cite{LMFDB}, if \(r=743, 761, 809, 821, 827, 1087, 1187, 1193\) that are not congruent to 1 modulo 12, then there is a newform orbit of size 41.
Since \(83 = 1 +2\times 41\), the odd number 83 is of the form \eqref{eq_18}.
Similarly, the odd number 123 can be written in the form \eqref{eq_18} for \(r=1277, 1289, 1627, 1663, 1709, 2039\).

Therefore, we shall propose the following conjecture.
\begin{conjecture}
For an arbitrary non-negative integer \(n\), there exists a prime number \(r\) such that there exists a subset \(W\subset \{1,\dotsc,s\}\) that gives an equality \(n = \sum_{k\smallin W}|O(k)|\), where \(O(1),\dotsc,O(s)\) are the newform orbits of \(S_2\big(\Gamma_0(r)\big)\).
\end{conjecture}
We expect that our theory can be extended to the case without the assumption \(r \equiv 1  \pmod{12}\).
If this extension is achieved and if our conjecture holds, then every odd number can be realized as Iwasawa \(\lambda\)-invariants of infinitely many constant \(\symbf{Z}_p\)-towers over DSIGs.
\begin{longtblr}[caption= Examples of odd numbers of the form \eqref{eq_18}, label=tab_1]{width=0.8825\textwidth, colspec={|c|c|c|X|}}
\hline
\(r\)&\SetCell[r=1]{m, 7em}number of newform orbits \(s\)&\SetCell[r=1]{m, 7em} {size of \\ newform orbits\\ \(|O(k)|\)} &odd numbers of the form \eqref{eq_18}
\\\hline
13&0& ---&1\\\hline
37&2&1, 1 &1, 3, 5 \\\hline
61&2&1, 3&1, 3, 7, 9 \\\hline
73&3&1, 2, 2&1, 3, 5, 7, 9, 11\\\hline
97&2&3, 4&1, 7, 9, 15 \\\hline
109&3&1, 3, 4&1, 3, 7, 9, 11, 15, 17 \\\hline
157&2&5, 7 &1, 11, 15, 25 \\\hline
181&2&5, 9&1, 11, 19, 29 \\\hline
193&3&2, 5, 8&1, 5, 11, 15, 17, 21, 27, 31\\\hline
229&3&1, 6, 11&1, 3, 13, 15, 23, 25, 35, 37\\\hline
241&2&7, 12&1, 15, 25, 39\\\hline
277&4&1, 3, 9, 9&1, 3, 7, 9, 19, 21, 25, 27, 37, 39, 43, 45\\\hline
313&3&2, 11, 12&1, 5, 23, 25, 27, 29, 47, 51\\\hline
337&2&12, 15& 1, 25, 31, 55\\\hline
349&2&11, 17&1, 23, 35, 57\\\hline
373&3&1, 12, 17&1, 3, 25, 27, 35, 37, 59, 61\\\hline
397&5&2, 2, 5, 10, 13&1, 5, 9, 11, 15, 19, 21, 25, 27, 29, 31, 35, 37, 39, 41, 45, 47, 51, 55, 57, 61, 65\\\hline
409&2&13, 20&1, 27, 41, 67\\\hline
421&2&15, 19&1, 31, 39, 69\\\hline
433&4&1, 3, 15, 16&1, 3, 7, 9, 31, 33, 35, 37, 39, 41, 63, 65, 69, 71\\\hline
457&3&2, 15, 20&1, 5, 31, 35, 41, 45, 71, 75\\\hline
541&2&20, 24&1, 41, 49, 89\\\hline
577&5&2, 2, 3, 18, 22&1, 5, 7, 9, 11, 15, 37, 41, 43, 45, 47, 49, 51, 53, 55, 59, 81, 85, 87, 89, 91, 95\\\hline
601&2&20, 29&1, 41, 59, 99\\\hline
613&3&5, 18, 27&1, 11, 37, 47, 55, 65, 91, 101\\\hline
661&3&2, 23, 29&1, 5, 47, 51, 59, 63, 105, 109\\\hline
673&4&2, 4, 24, 25&1, 5, 9, 13, 49, 51, 53, 55, 57, 59, 61, 63, 99, 103, 107, 111\\\hline
709&3&1, 27, 30&1, 3, 55, 57, 61, 63, 115, 117\\\hline
733&4&1, 2, 25, 32&1, 3, 5, 7, 51, 53, 55, 57, 65, 67, 69, 71, 115, 117, 119, 121\\\hline
757&2&29, 33&1, 59, 67, 125\\\hline
769&2&27, 36&1, 55, 73, 127\\\hline
829&3&1, 28, 39&1, 3, 57, 59, 79, 81, 135, 137\\\hline
853&2&33, 37&1, 67, 75, 141\\\hline
877&3&2, 32, 38&1, 5, 65, 69, 77, 81, 141, 145\\\hline
937&2&34, 43&1, 69, 87, 155\\\hline
997&8&1, 1, 1, 4, 5, 5, 23, 42&1, 3, 5, 7, 9, 11, 13, 15, 17, 19, 21, 23, 25, 27, 29, 31, 33, 35, 47, 49, 51, 53, 55, 57, 59, 61, 63, 65, 67, 69, 71, 73, 75, 77, 79, 81, 85, 87, 89, 91, 93, 95, 97, 99, 101, 103, 105, 107, 109, 111, 113, 115, 117, 119, 131, 133, 135, 137, 139, 141, 143, 145, 147, 149, 151, 153, 155, 157, 159, 161, 163, 165\\\hline
1009&2&37, 46&1, 75, 93, 167\\\hline
1021&2&37, 47&1, 75, 95, 169\\\hline
1033&2&40, 45&1, 81, 91, 171\\\hline
1069&4&2, 4, 31, 51&1, 5, 9, 13, 63, 67, 71, 75, 103, 107, 111, 115, 165, 169, 173, 177\\\hline
1093&3&3, 43, 44&1, 7, 87, 89, 93, 95, 175, 181\\\hline
1117&2&43, 49&1, 87, 99, 185\\\hline
1129&2&43, 50&1, 87, 101, 187\\\hline
1153&3&1, 44, 50&1, 3, 89, 91, 101, 103, 189, 191\\\hline
1201&3&2, 46, 51&1, 5, 93, 97, 103, 107, 195, 199\\\hline
1213&2&48, 52&1, 97, 105, 201\\\hline
1237&2&48, 54&1, 97, 109, 205\\\hline
1249&3&7, 37, 59&1, 15, 75, 89, 119, 133, 193, 207\\\hline
1297&3&1, 51, 55&1, 3, 103, 105, 111, 113, 213, 215\\\hline
1321&4&1, 3, 49, 56&1, 3, 7, 9, 99, 101, 105, 107, 113, 115, 119, 121, 211, 213, 217, 219\\\hline
1381&2&51, 63&1, 103, 127, 229\\\hline
1429&2&54, 64&1, 109, 129, 237\\\hline
1453&2&57, 63&1, 115, 127, 241\\\hline
1489&2&57, 66&1, 115, 133, 247\\\hline
1549&3&1, 59, 68&1, 3, 119, 121, 137, 139, 255, 257\\\hline
1597&2&63, 69&1, 127, 139, 265\\\hline
1609&3&2, 58, 73&1, 5, 117, 121, 147, 151, 263, 267\\\hline
1621&3&1, 63, 70&1, 3, 127, 129, 141, 143, 267, 269\\\hline
1657&3&2, 63, 72&1, 5, 127, 131, 145, 149, 271, 275\\\hline
1669&2&63, 75&1, 127, 151, 277\\\hline
1693&3&3, 65, 72&1, 7, 131, 137, 145, 151, 275, 281\\\hline
1741&2&66, 78&1, 133, 157, 289\\\hline
1753&3&2, 66, 77&1, 5, 133, 137, 155, 159, 287, 291\\\hline
1777&2&68, 79&1, 137, 159, 295\\\hline
1789&2&68, 80&1, 137, 161, 297\\\hline
1801&2&68, 81&1, 137, 163, 299\\\hline
1861&2&68, 86&1, 137, 173, 309\\\hline
1873&3&1, 75, 79&1, 3, 151, 153, 159, 161, 309, 311\\\hline
1933&3&1, 76, 83&1, 3, 153, 155, 167, 169, 319, 321\\\hline
1993&2&77, 88&1, 155, 177, 331\\\hline
\end{longtblr}
\begin{appendices}
\section{Level structure aspects of constant \(\symbf{Z}_p\)-towers}

In this appendix, we explain a natural interpretation of our constant \(\symbf{Z}_p\)-towers over DSIGs in terms of SIGs with \(\Gamma(p^n)\)-level structures.
A similar interpretation appears in \cite{LM_a}, \cite{LM_c}.

between SIGs with \(\Gamma(p^n)\)-level structures and our constant \(\symbf{Z}_p\)-towers over DSIGs.
Such a relation appears in \cite[\S 4]{LM_a}, but their situation is slightly different from ours.

As in section \ref{SI}, we choose a complete set \(\{E_0, E_1,\dotsc,E_h\}\) of representatives of the isomorphism classes of supersingular elliptic curves over \(\wideoverbar{\symbf{F}}_r\).
We recall supersingular isogeny graphs \(\SI(r,\ell,e_{p^n})\) with Weil pairing structure following \cite[\S 5]{LM_a}, \cite[\S 1]{LM_c}.
\begin{definition}[cf.\ {\cite[\S 5]{LM_a}, \cite[\S 1]{LM_c}}]
Let \(n\) be a non-negative integer, \(r\) be a prime number with \(r \equiv 1 \pmod{12}\), \(\ell\) be a prime number distinct from \(r\), and \(p\) be an odd prime number with \(p\not\mid r\ell\).
Let \(e_{p^n}\colon E[p^n]\times E[p^n]\to \symbf{\mu}(p^n)\) denote the Weil pairing on the subgroup of the \(p^n\)-torsion points of an elliptic curve \(E\), where \(\symbf{\mu}(p^n)\) is the group of \(p^n\)-th roots of unity.
The directed graph \(\SI(r,\ell,e_{p^n})\) is defined as follows.
Its set of vertices is given by
\[
\symbf{V}\big(\SI(r,\ell,e_{p^n})\big)\coloneq \bigcup_{i=0}^h\big\{\,(E_i,\zeta):\text{\(\zeta\in \symbf{\mu}(p^n)\) is primitive}\,\big\}.
\]
For each vertex \((E,\zeta)\in \symbf{V}\big(\SI(r,\ell,e_{p^n})\big)\), we fix a basis \((Q_1,Q_2)\) of \(E[p^n]\cong (\symbf{Z}/p^n\symbf{Z})\times(\symbf{Z}/p^n\symbf{Z})\) with \(e_{p^n}(Q_1,Q_2)=\zeta\).
We remark that a basis \((Q_1,Q_2)\) of \(E[p^n]\) is called a \emph{\(\Gamma(p^n)\)-level structure of \(E\).}
Then the set of directed edges from \((E,\zeta)\) to \((E',\zeta')\) is given by the \(\ell\)-isogenies \(\phi\colon E\to E'\) with \(e_{p^n}(\phi(Q_1),\phi(Q_2))=\zeta'\) up to automorphisms of \(E'\).
By bilinearity of Weil pairing, such a set of edges is independent of choice of basis.
Then we construct the graph \(Y^{(r,\ell)}_n\) from the directed graph \(\SI(r,\ell,e_{p^n})\) in the same manner as in (2) of Definitions \ref{def_1}.
\end{definition}
\begin{remark}
While the edges of \(\SI(r,\ell,e_{p^n})\) are given by \(\ell\)-isogenies up to automorphisms of their targets, the counterparts in \cite{LM_a}, \cite{LM_c} are given by \(\ell\)-isogenies \emph{not} up to automorphisms of their targets.
This difference necessitates a slight modification of the argument.
\end{remark}
Observe that, for each vertex \((E,\zeta)\) with a fixed basis \((Q_1,Q_2)\) and for each \(\ell\)-isogeny \(\phi\colon E \to E'\), the pair \((\phi(Q_1),\phi(Q_2))\) is a basis for \(E'[p^n]\), and we have
\begin{equation}\label{eq_22}
e_{p^n}(\phi(Q_1),\phi(Q_2)) \overset{\textcircled{\scriptsize\(\vardiamondsuit\)}}{=} e_{p^n}(Q_1,\hat{\phi}\circ \phi(Q_2)) = e_{p^n}(Q_1,\ell\cdot Q_2) = e_{p^n}(Q_1,Q_2)^{\ell} = \zeta^{\ell},
\end{equation}
and \(\zeta^{\ell}\) is a primitive \(p^n\)-th root of unity.
The equality under \textcircled{\scriptsize\(\vardiamondsuit\)} is the adjoint property for Weil pairings.
This implies that \(\phi\) induces a directed edge from \((E,\zeta)\) to \((E',\zeta^{\ell})\).
From this observation we claim the following.
\begin{proposition}[cf.\ {\cite[Proposition 5.2]{LM_a}}]\label{prop_a}
Let \(\beta_n\colon \symbf{E}\raisebox{.3ex}{\big(}X^{(r,\ell)}\raisebox{.3ex}{\big)}\to (\symbf{Z}/p^n\symbf{Z})^{\times}\) be the voltage assignment defined by \(\beta_n(e)=\ell\) for every edge \(e\in \symbf{E}\big(\SI(r,\ell)\big)\) and \(\beta_n(\bar{e})=\ell^{-1}\) for every edge \(\bar{e}\in \symbf{E}\big(\wideoverbar{\SI}(r,\ell)\big)\).
If we fix a primitive \(p^n\)-th root of unity \(\zeta_{p^n}\in \symbf{\mu}(p^n)\), then the bijective map \((\symbf{Z}/p^n\symbf{Z})^{\times} \to \{\,\zeta\in \symbf{\mu}(p^n):\text{\(\zeta\) is primitive}\,\};\ a \mapsto \zeta_{p^n}^a\) induces an isomorphism \(f\colon X^{(r,\ell)}(\beta_n)\cong Y^{(r,\ell)}_n\).
\end{proposition}
\begin{proof}
The map 
\[
f_{\symbf{V}}\colon \symbf{V}\big(X^{(r,\ell)}(\beta_n)\big)\to \symbf{V}\raisebox{.4ex}{\big(}Y^{(r,\ell)}_n\raisebox{.4ex}{\big)};\quad (E,a)\mapsto (E,\zeta_{p^n}^a)
\]
is a bijective map between the sets of vertices without a doubt.
Note that the sets \(\symbf{E}\big(\SI(r,\ell)\big)\) and \(\symbf{E}\big(\SI(r,\ell,e_{p^n})\big)\) are both subsets of the set of \(\ell\)-isogenies up to automorphisms of their targets.
So, by \eqref{eq_22}, for vertices \((E,a),(E',a')\in \symbf{V}\big(X^{(r,\ell)}(\beta_n)\big)\), 
\[
f_{\symbf{E}}\colon \symbf{E}\big(X^{(r,\ell)}(\beta_n)\big)_{(E,a),(E',a')}\to \symbf{E}\raisebox{.4ex}{\big(}Y^{(r,\ell)}_n\raisebox{.4ex}{\big)}_{(E,\zeta_{p^n}^a),(E',\zeta_{p^n}^{a'})}; \quad (e,a)\mapsto e
\]
is a well-defined bijective map.
Indeed, for compatibility of terminus maps, we can establish
\[
f_{\symbf{V}}(t([\phi],a))=f_{\symbf{V}}\big((E',a\ell)\big) = (E',\zeta_{p^n}^{a\ell}) \overset{\eqref{eq_22}}{=\joinrel =} t\big(f_{\symbf{E}}([\phi],a)\big),
\]
where \(\phi\colon E\to E'\) is an \(\ell\)-isogeny and \([\phi]\) denotes the class \(\{\,f\circ \phi: f\in \Aut(E')\,\}\).
The remaining conditions for being a morphism of graphs are straightforward.
\end{proof}
Proposition \ref{prop_a} tells us that morphisms \(Y^{(r,\ell)}_n \to X^{(r,\ell)}\) obtained by forgetting primitive \(p^n\)-th roots of unity form a compatible system 
\[
X^{(r,\ell)} = Y^{(r,\ell)}_0 \leftarrow Y^{(r,\ell)}_1 \leftarrow \cdots \leftarrow Y^{(r,\ell)}_n \leftarrow \cdots
\]
of covers with automorphism groups \((\symbf{Z}/p^n\symbf{Z})^{\times}\).
Now we reach the following statement.
\begin{proposition}[cf.\ {\cite[Lemma 4.5]{LM_c}}]\label{prop_b}
Assume that \(p\not\mid r\ell\) and \(\ell \equiv 1 \pmod{p}\).
Let \(n_0\) be the maximum integer among non-negative integers \(n\) with \(\ell\equiv 1 \pmod{p^n}\).
Let \(\gamma\colon \symbf{E}\raisebox{.25ex}{\big(}Y^{(r,\ell)}_{n_0}\raisebox{.25ex}{\big)} \to \symbf{Z}_p\) be the constant voltage assignment with value \((1/p^{n_0})\log_p(\ell) \in \symbf{Z}_p^{\times}\) on \(\symbf{E}\bigl(\SI(r,\ell,e_{p^{n_0}})\big)\) and multiplicative inverse on \(\symbf{E}\big(\wideoverbar{\SI}(r,\ell,e_{p^{n_0}})\big)\), where \(\log_p\) is the \(p\)-adic logarithm.
Then the following hold.
\begin{enumerate}[label=\((\arabic*)\)]
\item \(Y^{(r,\ell)}_{n_0}\) is the disjoint union of \(\varphi(p^{n_0})\)-copies of \(X^{(r,\ell)}\), where \(\varphi\) is Euler's totient function.
\item For each non-negative integer \(n\), the isomorphism in Proposition \ref{prop_a} induces an isomorphism \(Y^{(r,\ell)}_{n_0+n}\cong Y^{(r,\ell)}_{n_0}(\gamma_{n})\).
\end{enumerate}
\end{proposition}
\begin{proof}
(1)
We will deal with the derived graph \(X^{(r,\ell)}(\beta_n)\).
Since \(\ell\equiv 1 \pmod{p^{n_0}}\), two vertices \((E,a),(E',a')\) with \(a\neq a'\) are never connected by directed edges by definition of derived graph.
For each \(a \in (\symbf{Z}/p^{n_0}\symbf{Z})^{\times}\), the full subgraph of \(Y^{(r,\ell)}_{n_0}\) consisting of the vertices with \(a\) is nothing but \(X^{(r,\ell)}\).

\indent (2)
First we make an identification 
\[
(\symbf{Z}/p^{n_0+n}\symbf{Z})^{\times}\cong (\symbf{Z}_p/p^{n_0+n}\symbf{Z}_p)^{\times}\cong \symbf{Z}_p^{\times}/(1+p^{n_0+n}\symbf{Z}_p).
\]
Under this identification, the graph \(Y^{(r,\ell)}_{n_0}\) is an intermediate graph corresponding to the subgroup \((1+p^{n_0}\symbf{Z}_p)/(1+p^{n_0+n}\symbf{Z}_p)\) in the cover \(Y^{(r,\ell)}_{n_0+n}/X^{(r,\ell)}\).
Then, the cover \(Y^{(r,\ell)}_{n_0+n}/Y^{(r,\ell)}_{n_0}\) is given by the voltage assignment \(\delta_n\colon \symbf{E}\raisebox{.25ex}{\big(}Y^{(r,\ell)}_{n_0}\raisebox{.25ex}{\big)} \to (1+p^{n_0}\symbf{Z}_p)/(1+p^{n_0+n}\symbf{Z}_p)\) defined by \(\delta_n(e)=\ell\) for every edge \(e\in \symbf{E}\big(\SI(r,\ell,e_{p^{n_0}})\big)\); see Lemma \ref{lem_a} below for the general statement.
Consider the composition of isomorphisms
\begin{align}
\frac{1}{p^{n_0}}\log_p\colon &1+p^{n_0}\symbf{Z}_p \xrightarrow{\log_p}p^{n_0}\symbf{Z}_p\xrightarrow{p^{-n_0}\cdot(\cdot)}\symbf{Z}_p, \\
\frac{1}{p^{n_0}}\log_p\colon &1+p^{n_0+n}\symbf{Z}_p\xrightarrow{\log_p} p^{n_0+n}\symbf{Z}_p \xrightarrow{p^{-n_0}\cdot(\cdot)}p^{n}\symbf{Z}_p.
\end{align}
These isomorphisms induce the isomorphism 
\[
\frac{1}{p^{n_0}}\log_p\colon (1+p^{n_0}\symbf{Z}_p)/(1+p^{n_0+n}\symbf{Z}_p) \similarrightarrow \symbf{Z}_p/p^{n}\symbf{Z}_p,
\]
and thus we obtain isomorphisms \(Y^{(r,\ell)}_{n+n_0}\cong Y^{(r,\ell)}_{n_0}(\delta_n)\cong Y^{(r,\ell)}_{n_0}(\gamma_{n})\).
\end{proof}
\begin{lemma}\label{lem_a}
Let \(G\) be a finite group, \(X\) a finite graph, and \(\alpha\colon \symbf{E}(X)\to G\) a voltage assignment.
We do not impose connectedness on derived graph \(X(\alpha)\).
For an arbitrary subgroup \(H\subset G\), consider the action of \(H\) on \(X(\alpha)\) by left multiplication on \(G\).
Let \(Z\) be the quotient \(H\backslash X(\alpha)\) by the action \(H \curvearrowright X(\alpha)\).
Fix a complete set \(\{g_1,\dotsc,g_m\}\) of representatives of cosets \(H\backslash G\).
Define the voltage assignment \(\beta\colon \symbf{E}(Z)\to H\) by \((e,Hg_i)\mapsto g_i\alpha (e)g_j^{-1}\) where \(j\) is the unique number such that \(g_i\alpha(e)\in Hg_j\).
Then the cover \(X(\alpha)/Z\) is isomorphic to the cover \(Z(\beta)/Z\).
\end{lemma}
\begin{proof}
Recall that the graph \(Z\) is defined as follows.
The sets of vertices and edges are given by \(\symbf{V}(Z)\coloneq \symbf{V}(X)\times (H\backslash G)\), \(\symbf{E}(Z)\coloneq \symbf{E}(X)\times (H\backslash G)\) respectively.
For an edge \((e,Hg)\in \symbf{E}(Z)\), its origin, terminus, and inverse edge are given by \(o(e,Hg)\coloneq (o(e),Hg)\), \(t(e,Hg)\coloneq \big(t(e),Hg\alpha(e)\big)\), \((e,Hg)\,\bar{}\,\coloneq \big(\bar{e},Hg\alpha(e)\big)\) respectively.
Note that using fixed representatives \(g_1,\dotsc,g_m\) of \(H\backslash G\), we have \((e,Hg_i)\,\bar{}\,=(\bar{e},Hg_j)\) if \(g_i\alpha(e) \in Hg_j\).
This allows us to calculate \(\beta(\bar{e},Hg_j)=g_j\alpha(\bar{e})g_i^{-1}=\beta(e,Hg_i)^{-1}\), which means that \(\beta\) is indeed a voltage assignment.
The derived graph \(Z(\beta)\) can be described as
\begin{align}
\symbf{V}(Z(\beta))&=\{\,(v,Hg_i,h):v\in \symbf{V}(X),\ i\in\{1,\dotsc,m\},\ h\in H\,\},\\
\symbf{E}(Z(\beta))&=\{\,(e,Hg_i,h):e\in \symbf{E}(X),\ i\in\{1,\dotsc,m\},\ h\in H\,\},
\end{align}
as the sets of vertices and edges, and for each edge \((e,Hg_i,h)\), if \(g_i\alpha(e)\in Hg_j\), its origin, terminus, and inverse edge are
\begin{align}
o(e,Hg_i,h) & = (o(e),Hg_i,h), \\
t(e,Hg_i,h) &= \big(t(e),Hg_j,hg_i\alpha(e)g_j^{-1}\big), \\
(e,Hg_i,h)\,\bar{}\, &= \big(\bar{e},Hg_j,hg_i\alpha(e)g_j^{-1}\big)
\end{align}
respectively.
By this description, we can verify that the bijective map
\[
H\backslash G \times H \to G\,;\quad (Hg_i,h) \mapsto hg_i
\]
induces an isomorphism \(Z(\beta)\similarrightarrow X(\alpha)\).
\end{proof}
Since a connected component of \(Y^{(r,\ell)}_{n_0}\) is canonically isomorphic to \(X^{(r,\ell)}\) thanks to (1) of Proposition \ref{prop_b}, restricting the constant voltage assignment \(\gamma\) to each connected component of \(Y^{(r,\ell)}_{n_0}\) gives a constant \(\symbf{Z}_p\)-tower over \(X^{(r,\ell)}\).
Hence, Proposition \ref{prop_b} asserts that, if \(p \neq r\), the tower
\[
Y^{(r,\ell)}_{n_0} \leftarrow Y^{(r,\ell)}_{n_0+1} \leftarrow \cdots \leftarrow Y^{(r,\ell)}_n \leftarrow \cdots
\]
is ``the disjoint union of \(\varphi(p^{n_0})\)-copies of the constant \(\symbf{Z}_p\)-tower over \(X^{(r,\ell)}\).''
\end{appendices}
%%%%%%%%%%%%%%%%%%%%%%%%%%%%%%%
\begin{center}
\LARGE\decotwo
\end{center}


\begin{thebibliography}{99}
\bibitem{AMT} Adachi, T., Mizuno, K., \& Tateno, S. (2026). Iwasawa theory for weighted graphs. {\it Ann.\ Math.\ Qu\'ebec}, {\it 50}, 231--265. %\url{https://doi.org/10.1007/s40316-025-00247-w}
\bibitem{BK} Birch, B. J., \& Kuyk, W. (Eds.). (1972). {\it Modular Functions of One Variable IV}. Springer Berlin, Heidelberg. %\url{https://doi.org/10.1007/BFb0097580}
\bibitem{Del} Deligne, P. (1974). La conjecture de Weil.\ I. {\it Publications Math\'ematiques de L'Institut des Hautes \'Etudes Scientifiques}, {\it 43}, 273--307. %\url{https://doi.org/10.1007/BF02684373}
\bibitem{DS_74} Deligne, P., \& Serre, J.-P. (1974). Formes modulaires de poids 1. {\it Ann.\ Sci.\ \'Ecole Norm.\ Sup., 4}, 507--530. %\url{https://doi.org/10.24033/asens.1277}
\bibitem{DS_05} Diamond, F., \& Shurman, J. (2005). {\it A First Course in Modular Forms}. Springer New York. %\url{https://doi.org/10.1007/978-0-387-27226-9}
\bibitem{DLRV} Dion, C., Lei, A., Ray, A., \& Valli\`eres, D. (2024). On the distribution of Iwasawa invariants associated to multigraphs. {\it Nagoya Math.\ J.}, {\it 253}, 48--90. %\url{https://doi.org/10.1017/nmj.2023.18}
\bibitem{DV} DuBose, S., \& Valli\`eres, D. (2023). On \(\symbb{Z}_{\ell}^d\)-towers of graphs. {\it Algebr.\ Comb., 6}(5), 1331--1346. %\url{https://doi.org/10.5802/alco.304}
\bibitem{Eic_a} Eichler, M. (1973). The Basis Problem for Modular Forms and the Traces of the Hecke Operators. In Kuijk, W. (Eds.)., {\it Modular Functions of One Variable I. Lecture Notes in Mathematics, 320}, 75--151. Springer Berlin, Heidelberg. %\url{https://doi.org/10.1007/978-3-540-38509-7_4}
\bibitem{Eic_b} Eichler, M. (1977). On theta functions of real algebraic number fields. {\it Acta Arithmetica 33}(3), 269--292. %\url{https://doi.org/10.4064/aa-33-3-269-292}
\bibitem{Gon} Gonet, S. R. (2022). Iwasawa theory of Jacobians of graphs. {\it Algebr.\ Comb., 5}(5), 827--848. %\url{https://doi.org/10.5802/alco.225}
\bibitem{GL} Goren, E.\ Z., \& Love, J.\ (2024). Supersingular elliptic curves, quaternion algebras and applications to cryptography. {\it NATO Science for Peace and Security Series - D: Information and Communication Security, 66}, 123--200. %\url{https://doi.org/10.3233/NICSP250005}
\bibitem{Gro} Gross, B. H. (1986). Heights and the Special Values of L-series. {\it CMS Conf.\ Proceedings, Vol.\ 7, AMS}, 115--187.
\bibitem{Har} Hartshorne, R. (1977). {\it Algebraic Geometry}. Springer New York. %\url{https://doi.org/10.1007/978-1-4757-3849-0}
\bibitem{Iwasawa} Iwasawa, K. (1959). On \(\Gamma\)-extensions of algebraic number fields. {\it Bull.\ Amer.\ Math.\ Soc., 65}, 183--226. %\url{https://doi.org/10.1090/S0002-9904-1959-10317-7}
\bibitem{Kataoka} Kataoka, T. (2026). {\it Construction of graph coverings with prescribed Iwasawa invariants}. arXiv preprint, \href{https://doi.org/10.48550/arXiv.2603.23088}{\tt arXiv:2603.23088v1 [math.CO]}.
\bibitem{KM} Kleine, S., \& M\"uller, K. (2024). On the growth of the Jacobians in \(\symbb{Z}_p\)-voltage covers of graphs. {\it Algebr.\ Comb.}, {\it 7}(4), 1011--1038. %\url{https://doi.org/10.5802/alco.366}
\bibitem{LLZ} Lei, A., Loeffler, D., \& Zerbes, S. L. (2014). Euler systems for Rankin-Selberg convolutions of modular forms. {\it Annals of Mathematics}, {\it 180}, 653--771. %\url{https://doi.org/10.4007/annals.2014.180.2.6}
\bibitem{LM_a} Lei, A., \& M\"uller, K.\ (2025). On towers of isogeny graphs with full level structures. {\it Res.\ Math.\ Sci.}, {\it 12}, Article 4. %\url{https://doi.org/10.1007/s40687-024-00478-3}
\bibitem{LM_b} \rule[-0.125ex]{3em}{0.5pt} (2025). {\it Isogeny graphs with level structures arising from the Verschiebung map}. arXiv preprint, \href{https://doi.org/10.48550/arXiv.2501.03846}{\tt arXiv:2501.03846v1 [math.NT]}. To appear in the proceedings of the ICTS program ``Elliptic curves and the special values of \(L\)-functions''.
\bibitem{LM_c} \rule[-0.125ex]{3em}{0.5pt} (2025). {\it On \(\symbb{Z}_p\)-towers of graph coverings arising from a constant voltage assignment}. arXiv preprint, \href{https://doi.org/10.48550/arXiv.2501.03852}{\tt arXiv:2501.03852v2 [math.NT]}. To appear in Glasgow Mathematical Journal.
\bibitem{LM_d} \rule[-0.125ex]{3em}{0.5pt} (2026). {\it Iwasawa theory for abelian towers of digraphs}. arXiv preprint, \href{https://doi.org/10.48550/arXiv.2601.19571}{\tt arXiv:2601.19571v1 [math.NT]}.
\bibitem{Loe} Loeffler, D. (2017). Images of adelic Galois representations for modular forms. {\it Glasgow Mathematical Journal}, {\it 59}(1), 11--25. %\url{https://doi.org/10.1017/S0017089516000367}
\bibitem{MV_24} McGown, K., \& Valli\`eres, D. (2024). On abelian \(\ell\)-towers of multigraphs III. {\it Ann.\ Math.\ Qu\'ebec}, {\it 48}, 1--19. %\url{https://doi.org/10.1007/s40316-022-00194-w}
\bibitem{Momose} Momose, F. (1981). On \(\ell\)-adic representations attached to modular forms, {\it J. Fac.\ Sci.\ Univ.\ Tokyo Sec.\ IA Math.}, {\it 28}(1), 89--109.
\bibitem{MS_23} Munier, N., \& Shnidman, A. (2023). Sandpile groups of supersingular isogeny graphs, {\it Journal de Th\'eorie des Nombres de Bordeaux}, {\it 35}, 751--774.
\bibitem{MT} Murooka, R., \& Tateno, S. (2025). {\it Iwasawa theory for vertex-weighted graphs}. arXiv preprint \href{https://doi.org/10.48550/arXiv.2505.12351}{\tt arXiv:2505.12351v1 [math.CO]}.
%\bibitem{Neu} Neukirch, J.\ (1999). {\it Algebraic Number Theory}. Springer Berlin, Heidelberg. %\url{https://doi.org/10.1007/978-3-662-03983-0}
%\bibitem{Ribet_77} Ribet, K. (1977). Galois representations attached to eigenforms with nebentypus. In Serre, J.-P., Zagier, D.B. (Eds.)., {\it Modular Functions of one Variable V. Lecture Notes in Mathematics, 601}, 17--52. %\url{https://doi.org/10.1007/BFb0063943}
\bibitem{Ribet_80} Ribet, K. (1980). Twists of Modular Forms and Endomorphisms of Abelian Varieties. {\it Mathematische Annalen}, {\it 253}, 43--62. %\url{https://doi.org/10.1007/BF01457819}
\bibitem{Ribet_85} \rule{3em}{0.5pt} (1985). On l-adic representations attached to modular forms II. {\it Glasgow Mathematical Journal}, {\it 27}, 185--194. %\url{https://doi.org/10.1017/S0017089500006170}
\bibitem{Serre} Serre, J.-P. (1980). {\it Trees}. Springer Berlin, Heidelberg. %\url{https://doi.org/10.1007/978-3-642-61856-7}
\bibitem{Silv} Silverman, J. H. (2009). {\it The Arithmetic of Elliptic Curves} (2nd ed.). Springer New York. %\url{https://doi.org/10.1007/978-0-387-09494-6}
\bibitem{LMFDB} The LMFDB Collaboration (2026). {\it The L-functions and modular forms database} [Online; accessed 7 April 2026]. %\url{https://www.lmfdb.org}
\bibitem{Terras} Terras, A. (2010). {\it Zeta Functions of Graphs: A Stroll through the Garden}. Cambridge University Press. %\url{https://doi.org/10.1017/CBO9780511760426}
\bibitem{sage} The Sage Developers (2025). {\it SageMath, the Sage Mathematics Software System} (Version 10.8) [Computer software]. %\url{https://www.sagemath.org}
\bibitem{Vall} Valli\`eres, D. (2021). On abelian \(\ell\)-towers of multigraphs. {\it Ann.\ Math.\ Qu\'e.}, {\it 45}(2), 433--452. %\url{https://doi.org/10.1007/s40316-020-00152-4}
\bibitem{Velu} V\'elu, J. (1971). Isog\'enies entre courbes elliptiques. {\it Comptes rendus de l'Acad\'emie des Sciences de Paris, S\'erie A}, {\it 273}(4), 238--241. %\url{https://gallica.bnf.fr/ark:/12148/bpt6k56191248}
\end{thebibliography}
\end{document}